\documentclass[11pt,english]{amsart}
\usepackage[T1]{fontenc}
\usepackage[english]{babel}
\usepackage[a4paper]{geometry}

\usepackage{amsmath}
\usepackage{amssymb}
\usepackage{amstext}
\usepackage{amsthm}
\usepackage{color}
\usepackage{url}

\newtheorem{thm}{Theorem}[section]
\newtheorem{cor}[thm]{Corollary}
\newtheorem{lem}[thm]{Lemma}
\newtheorem{prop}[thm]{Proposition}

\newtheorem{rem}[thm]{Remark}
\newtheorem{example}[thm]{Example}
\newtheorem{assumptions}[thm]{Assumptions}

\numberwithin{equation}{section}

\newcommand{\R}{\mathbb{R}}

\newcommand{\E}{\mathbb{E}}
\newcommand{\Ind}{{\rm{\textbf{1}}}}

\newcommand{\rev}[1]{{#1}}

\title[Systems smoothly interacting through hitting times]{Mean-field limit of a stochastic particle system smoothly interacting through threshold hitting-times and applications to neural networks with dendritic component.}

\author{J. Inglis and D. Talay}
\address{Equipe Tosca, INRIA Sophia-Antipolis M\'editerran\'ee,
2004 route des lucioles, BP 93,  06902 Sophia Antipolis Cedex,  France} 
\email{james.inglis@inria.fr, denis.talay@inria.fr}

\thanks{This work was partially supported by the European Union Seventh Framework Programme (FP7) under grant agreement no. 269921 (BrainScaleS), no. 318723 (Mathemacs) and the Human Brain Project (HBP)}

\begin{document}
\maketitle
\begin{abstract}
In this article we study the convergence of a stochastic particle system that interacts through threshold hitting times towards a novel equation of McKean-Vlasov type.  The particle system is motivated by an original model for the behavior of a network of neurons, in which a classical noisy integrate-and-fire model is coupled with a cable equation to describe the dendritic structure of each neuron. 

\vspace{5pt}
\noindent {\sc Keywords:} McKean-Vlasov equation; cable equation; noisy integrate-and-fire model; mean-field limit; non-constant synaptic weights; dendritic structure; neuroscience.
\end{abstract}

\section{Introduction}

Particle systems that interact through hitting/exit times are receiving an increasing amount of attention in diverse areas.  For some recent references and a discussion of computational models in fluid and population dynamics, see \cite[Sec.5]{TalayICM}.  Other examples can be found in financial mathematics: in \cite{Giesecke1} such a system is used to model portfolio default rates.  
In this article our motivation comes principally from neuroscience, though the new model we propose may have different applications (see Section \ref{conclusion} below).  The specific nature of the interactions in the model requires the development of some new analytical techniques in order to study the mean-field limit.

Recently, several works have been concerned with a particular model of neuronal activity that consists of a network of $N$ neurons each evolving according to the so-called noisy integrate-and-fire model, and interacting though an empirical average of jump terms.  To be more precise, the framework introduced in \cite{ostojic:brunel:hakim} modeled the electrical potential $U^i_t$  across the soma of neuron $i$ at time $t$ as a stochastic process $U^i = (U^i_t)_{t\geq0}$ with the following two key properties: 
\begin{itemize}
\item[(i)] whenever $U^i_t$ reaches a constant threshold, it is instantaneously reset to some value below the threshold (at such times the neuron instigates an action potential and is said to have `spiked');
\item[(ii)] at spike times all the other neurons instantaneously receive a `kick' of size proportional to $1/N$, which may in turn cause them to spike if they are close enough to the threshold.
\end{itemize}
The idea is then to approximate the interaction described by (ii) by an average when $N$ becomes large i.e. to take the mean-field limit.  The resulting non-standard equation of McKean-Vlasov type is the subject of \cite{DIRT} from a stochastic process point of view, and \cite{CCP, CGGS} from a PDE perspective. In particular, it was shown in these works that for certain choices of parameters the limit equation has a unique global-in-time solution, while for others it exhibits a blow-up phenomenon in finite time (where the effect of a single neuron spiking causes an instantaneous cascade during which a macroscopic proportion of other neurons all spike at exactly the same time).  This causes serious problems when trying to approximate the behavior of the finite system by the limit equation after a blow-up (though this problem has been partially resolved in \cite{DIRT2}).

A major criticism of the \rev{aforementioned} model is the instantaneous transmission of the effect of a spiking neuron to all points in the network.  This is because in reality it takes time for the arrival of an action potential at a synapse to be felt at the soma, as it is transmitted along the dendritic tree.  The first purpose of this article is thus to introduce a new model that takes into account this effect, by describing how the transmission occurs through the use of the cable equation.

It turns out that by doing this, we have a rare example whereby an attempt to derive a model that better reflects reality in fact helps with the analysis in some places, and allows us to make significant generalizations. 
The two key generalizations we make here in comparison with previous setups are that we allow for non-constant synaptic weights (so that the effect of neuron $i$ on neuron $j$ is not necessarily the same as the effect of $i$ on $k$), and that in place of a constant diffusion coefficient we work with a much more general (elliptic and bounded) one.  Both these generalizations are important from the modeling point of view, since non-constant synaptic weights and general noise are important features in neural networks.
 
The advantage of introducing the cable equation as a transmission mechanism is that the dynamics at the soma are smoothed.  As we will see, this eliminates the possibility of a blow-up cascade, which is an artifact of the instantaneous transmission of action potentials in previous models.   Nevertheless, the resulting McKean-Vlasov equation and associated nonlinear martingale problem is of a new type, and the analysis of its solution requires arguments which are not available in the literature.   

The main results (Theorem \ref{thm: existence and uniqueness} and Theorem \ref{thm: convergence}) of the current article are that in the limit as the size of the network becomes infinite, our new model can indeed be approximated by the unique solution to the associated McKean-Vlasov type limit equation.  The general proof strategy of these two results takes inspiration from that presented in \cite{DIRT} and  \cite{DIRT2}, though it has to be significantly adapted to the new situation.  Indeed we develop new arguments that both reduce the complexity of the proofs found in these previous works by taking advantage of the smoothing effect of the transmission mechanism, and allow us to handle the non-constant synaptic weights and general diffusion term.

The layout of the paper is as follows.  In Section \ref{sec: main results}, we present the particle system and nonlinear McKean-Vlasov equation under study, together with the main results.  In Section \ref{sec: derivation} we derive the model from biological considerations.  In Section \ref{sec: e+u} the first main result (Theorem \ref{thm: existence and uniqueness}) of the existence and uniqueness of a solution to the nonlinear McKean-Vlasov equation is proved.  In Section \ref{sec:convergence} the convergence of the particle system (Theorem \ref{thm: convergence}) is proved, and we finish with a conclusion in Section \ref{conclusion}.

\vspace{0.5cm}
\noindent\textit{General notation:} For a \rev{subset of Euclidean} space $\mathcal{T}$, we will denote by $\mathcal{C}(\mathcal{T})$ the space of all continuous functions $:\mathcal{T}\to\R$, equipped with the usual uniform norm $\|f\|_\infty := \sup_{t\in\mathcal{T}}|f(t)|$ for  $f\in\mathcal{C}(\mathcal{T})$. Where necessary we will write $\mathcal{C}_b(\mathcal{T})$ for the space of all bounded continuous functions on $\mathcal{T}$.
For $n\in\{1, \dots, \infty\}$, $\mathcal{C}^n(\mathcal{T})$ will denote the space of $n$ times continuously differentiable functions $:\mathcal{T}\to\R$, equipped with the norm $\|f\|_{\mathcal{C}^n(\mathcal{T})} := \sum_{i=0}^n \|f^{(i)}\|_\infty$ for  $f\in\mathcal{C}^n(\mathcal{T})$, 
where $f^{(i)}$ indicates the $i$th derivative of $f$.  Similarly, $\mathcal{C}^n_b(\mathcal{T})$ will denote the subspace of bounded functions in $\mathcal{C}^n(\mathcal{T})$.
For $f\in\mathcal{C}([0, \infty))$ we will sometimes write $\|f\|_{\infty, T} := \sup_{t\in[0, T]}|f(t)|$ for $T>0$.
Moreover, for $T>0$, we will use $\mathcal{D}([0, T])$ to denote the Skorohod space of c\`adl\`ag functions, equipped with the usual Skorohod topology.
We will also write $\mathcal{P}(\mathcal{T})$ for the space of probability measures on $(\mathcal{T}, \mathcal{B}(\mathcal{T}))$, where $\mathcal{B}(\mathcal{T})$ is the set of Borel sets of $\mathcal{T}$.  Finally, we use the notation $\lfloor x\rfloor$ to represent the integer part of $x\geq0$, and include a nomenclature table in the Appendix (Section \ref{nomenclature}) to help the reader keep track of the more specific definitions.

\section{Main results}
\label{sec: main results}
The central object of study in this article will be an $N$-particle system interacting through threshold hitting times.
The motivation for studying this particular particle system will be described in detail in the next section.  In short, the system models the behavior of the electrical potential $(U_t^i)_{t\geq0}$ across the soma of neuron $i$ in a neural network of size $N$. Each time this potential reaches the threshold~1 (at times $(\tau_k^i)_{k\geq1}$) the potential is instantaneously reset to $0$.  At such a time the neuron is said to `spike'.  The neurons in the network interact through a term that describes the effect of a spiking neuron on the rest of the network, with the coefficients $J_{ij}$ representing the connection strengths, and the kernel $G$ modeling the mechanism whereby a spike is transmitted from a synapse (where it is received) along the dendritic tree to the soma.

Precisely, the system is given by
\begin{equation}
\begin{cases}
U_t^i=U^i_0 + H(t)  + \int_0^t b(U_s^i)ds +\sum_{j=1}^N\frac{J_{ij}}{S^N_i}\int_0^t   {G}(t -s)M_s^j ds - M^i_t + \int_0^t\sigma(U_s^i)dW_s^i,\\
S^N_i := \sum_{j=1}^NJ_{ij}, \quad M_t^i := \sum_{k=1}^\infty \Ind_{[0,t]}(\tau_k^i),\\
\tau_k^i:= \inf\{t\geq\tau_{k-1}^i:  U_{t-}^i \geq 1\},\ k\in\mathbb{N}\backslash\{0\}, \quad \tau_0^i = 0,
\end{cases}
\label{particle system}
\end{equation}
for $t\geq0$ and $i\in\{1, \dots, N\}$.  Here $(W^i_t)_{t\geq0}$ are independent standard one-dimensional Brownian motions and the weights $J_{ij}$ are non-negative constants.  The integer-valued processes $(M_t^i)_{t\geq0}$ count the number of times the potential $U^i$ has reached the threshold before $t$.  From a biological point of view, they count the number of times the neuron has `spiked' before $t$.  

To understand the system \eqref{particle system}, note that in between system spiking times, all the potentials $(U^i)_{i\in\{1, \dots, N\}}$ simply follow an SDE \rev{whose coefficients do not depend on the behavior of the other neurons during this period, since over such a time interval all the $(M^i)_{i\in\{1, \dots, N\}}$ are constant. That is, given the value of $U^i$ at a spiking time, its evolution in time up until the next system spiking time will be independent of the behavior of the other potentials during this interval}.  Then, when the system spikes i.e. when one of the potentials $U^i$ reaches the threshold, the SDEs are all updated by setting $M^i_{t} = M^i_{t-} +1$. The process $(M^i_t)_{t\geq0}$ (and hence $(U^i_t)_{t\geq0}$) is c\`adl\`ag, so that $U^i_t<1$ almost surely. 

The central goal is the study of the convergence of the particle system \eqref{particle system} as $N\to\infty$ towards the limit equation
\begin{equation}
\begin{cases}
U_t &= U_0 + H(t)  + \int_0^t b(U_s)ds +\int_0^t   {G}(t -s)\E(M_s) ds - M_t + \int_0^t\sigma(U_s)dW_s,\\
M_t &:= \sum_{k=1}^\infty \Ind_{[0,t]}(\tau_k),\\
\tau_k &=: \inf\{t\geq\tau_{k-1}:  U_{t-}\geq 1\},\ k\in\mathbb{N}\backslash\{0\}, \quad \tau_0 = 0.
\end{cases}
\label{limit equation}
\end{equation}
This is a non-trivial equation of McKean-Vlasov type, since the right-hand side depends on the distribution of the solution.  However, it is non-standard since it in fact depends on the distribution of the \textit{hitting times} of the threshold by the solution, rather than the solution itself.

Throughout the article, we will impose the following conditions on the coefficients $H, b, \sigma, G$ and the initial condition $U_0$:
\begin{assumptions}
\label{assumptions}
Assume that
\begin{itemize}
\item $H:[0, \infty)\to\R$ is bounded and twice differentiable, with bounded derivatives;
\item $b\in\mathcal{C}^1(\R)$ \rev{with bounded derivative, and set $\Lambda_b:= \max\{\|b'\|_\infty, b(0)\}$, so that}
\[
|b'(x)|\leq \Lambda_b, \quad |b(x)|\leq \Lambda_b(1+|x|), \quad x\in\R;
\] 
\item $\sigma\in\mathcal{C}^2_b(\R)$ and there exists a constant $\Lambda_\sigma>0$ such that
\[
\Lambda_\sigma^{-1} \leq \sigma(x) \leq \Lambda_\sigma, \quad |\sigma'(x)|, |\sigma''(x)|\leq \Lambda_\sigma, \quad x\in\R;
\]
\item $G:[0, \infty)\to\R$ is bounded and twice differentiable, with bounded derivatives, such that 
\[
G(0) = G'(0) =0;
\]
\item $U_0\in(-R, 1)$ almost surely for some $R\geq1$, and is such that
\begin{equation}
\label{initial condition}
\sup_{t\in(0,T]}\int_{-\infty}^1 (1-x)t^{-\frac{3}{2}}e^{-\frac{(1-x)^2}{\Lambda_\sigma t}}\mathbb{P}(U_0 \in dx) < \infty
\end{equation}
for all $T>0$.
\end{itemize}
\end{assumptions}

\begin{rem}[The initial condition]
The condition \eqref{initial condition} on $U_0$ ensures that the density of the first hitting time $\tau_1$ in \eqref{limit equation} does not explode as $t\to0$ (see Lemma \ref{solution fixed point}).  It is natural, since the same condition is necessary to ensure that the first hitting time of $1$ by a standard Brownian motion starting from $U_0$ has a bounded density as $t\to0$.  

One can also note that, for example, \eqref{initial condition} is certainly satisfied if $U_0 = x_0$ for some $x_0<1$, or if $U_0$ has a density with respect to the Lebesgue measure that is bounded and such that 
\begin{equation*}
\mathbb{P}(U_0 \in dx) \leq \beta(1-x)dx, \qquad x\in[1-\epsilon, 1),
\end{equation*}
for some $\beta\geq0$ and $\epsilon>0$.

The fact that we require the density of $U_0$ to have compact support is for convenience, and is not necessary for our results in some cases of interest. For example if $b(x) = -\lambda x$ for $x, \lambda\in \R$, then it could be replaced by requiring all moments to be finite.  The difficulty for general $b$ satisfying the above conditions is to arrive at a satisfactory estimate of the density of the first hitting time of the threshold $1$ by a stochastic process with drift $b$ (see the Appendix).  The reason we choose to use work with $U_0$ in a compact set is to avoid diverting attention away from the core features of the model, as highlighted in Remark \ref{Comparison with old model}.  We instead include further discussion of this issue in the Appendix (Section \ref{Improvements}).
\end{rem}

The first main result is an existence and uniqueness result for a solution to the nonlinear equation \eqref{limit equation}. 

\begin{thm}
\label{thm: existence and uniqueness}
Under the Assumptions \ref{assumptions} there exists a pathwise unique solution $U = (U_t)_{t\in[0, T]}$ to the equation \eqref{limit equation}, for any $T>0$.
\end{thm}

The second result is that, under some conditions on the weights $J_{ij}$, we have the convergence of the weighted empirical measure towards the limit equation.

\begin{thm}
\label{thm: convergence}
Let $(U^i_0)_{i\geq1}$ be a family of independent identically distributed random variables with the same law as $U_0$, and let $(U^1, \dots, U^N)$ be the solution to the particle system \eqref{particle system} on $[0,T]$, $T\geq0$.  Moreover, suppose that for all $i\in\mathbb{N}$
\begin{equation}
\label{J condition}
\frac{\sum_{j=1}^NJ^2_{ij}}{(S^i_N)^2} \to 0 \quad \mathrm{as}\ N\to\infty.
\end{equation}
Then, under the Assumptions \ref{assumptions}, for every $i\in\mathbb{N}$,
\[
\mathrm{Law}\left(\frac{1}{S^N_i}\sum_{\rev{j=1}}^NJ_{ij}\delta_{U^j}\right) \Rightarrow \delta_{\mathrm{Law}((U_t)_{t\in[0, T]})},
\]
in the space $\mathcal{P}(\mathcal{P}(\mathcal{D}([0, T])))$, where $(U_t)_{t\in[0, T]}$ is the unique solution to the limit equation \eqref{limit equation} on $[0, T]$.  Here $\mathcal{D}([0, T])$ denotes the Skorohod space on $[0, T]$, and $\Rightarrow$ indicates that the convergence is in the weak sense.
\end{thm}

\begin{rem}[Comparison with previous models]
\label{Comparison with old model}
The nonlinear equation \eqref{limit equation} and particle system \eqref{particle system} are similar to the ones studied in \cite{DIRT}  and \cite{DIRT2} respectively (and also in \cite{CCP} and \cite{CGGS} from a PDE viewpoint).  Indeed, one could recover the equation and particle system considered therein by formally taking $G = \delta_0$, $H(t) \equiv0$, $\sigma\equiv 1$ and $J_{ij}=1$ for all $i, j$.  However, the conditions we here impose on $G$ prevent this: the point is that they ensure that the discontinuities of the counting processes $M$ and $M^i$ are smoothed.  
From a biological standpoint, the reason for the presence of the kernel $G$ is described in more detail in Section \ref{sec: derivation}, but loosely it describes the fact that an arriving spike from a presynaptic neuron should be smoothly transferred from the synapse to the soma, and that there are no synapses on the soma itself (so that there is no instantaneous spike transmission).  In this respect our model is in fact more biologically realistic than the previously considered one.

Mathematically, the extra regularity due to the smoothing effect of $G$ renders our analysis more tractable, and allows us to make two important generalizations in comparison with the previously studied model: 1) we can include non-constant synaptic weights $J_{ij}$, and 2) we treat the case of a general non-constant diffusion coefficient $\sigma$.

A final remark is that by Theorem \ref{thm: existence and uniqueness}, we see that the limit equation \eqref{limit equation} does not exhibit the blow-up phenomenon observed for the previous model in \cite{CCP}.
 \end{rem}
 
The strategy to prove these two results is actually quite standard, though the details are not.  For Theorem \ref{thm: existence and uniqueness} we develop a fixed point argument taking inspiration from the classical notes of Sznitman \cite{Sznitman}, which concern particle systems with regular drift and diffusion terms.  The difficulty in our case stems from the fact that the interactions we consider are very singular: the particles in \eqref{particle system} interact through their hitting times of the threshold $1$.  The proof of Theorem \ref{thm: convergence} then proceeds in two steps: we first show that the family of laws of the weighted empirical measures is tight, and then prove that the law under any limit point of this tight sequence must be the law of a solution to the limit equation.  Since such a law is unique by Theorem \ref{thm: existence and uniqueness}, the result follows.  Once again this step is significantly complicated in comparison with the classical case due to the singular nature of the interactions.

\begin{example}
The condition \eqref{J condition} on the $J_{ij}$s is both satisfied in the classical case where $J_{ij} = 1$ for all $i, j\in\{1, \dots, N\}$, as well as in more interesting cases such as 
\[
J_{ij} = \frac{1}{|i-j|},\ j\neq i, \qquad J_{ii} = 0,
\]
for $i, j\in\{1, \dots, N\}$.  The latter case describes a situation where one has a decay in the synaptic strength between neurons that are far away from each other.
\end{example}

We finally mention a corollary to  Theorem \ref{thm: convergence}, which applies in the particular case when the system is exchangeable.  This is the standard propagation of chaos property, which follows by slightly adapting \cite[Proposition 2.2]{Sznitman}.

\begin{cor}
Suppose that we are in the situation of Theorem \ref{thm: convergence}, but that in addition we have that $J_{ij}$ does not depend on $i$, for all $i,j\in\{1, \dots, N\}$.  Then for any $\phi_1, \dots, \phi_k \in \mathcal{C}_b(\mathcal{D}([0, T]))$ it holds that
\[
\E\left(\phi_1(U^1)\dots \phi_k(U^k)\right) \to \prod_{l=1}^k\E(\phi_l(U)), \quad \mathrm{as}\ N\to\infty.
\]

\end{cor}

\section{Derivation of the particle system \eqref{particle system}}
\label{sec: derivation}
In this section we motivate the study of the system \eqref{particle system} as a model for the behavior of a network of $N$ neurons, labelled $i\in\{1, \dots, N\}$ (though as mentioned in Section \ref{conclusion}, there may exist other applications motivating similar models). We suppose that each neuron consists of a soma and a dendritic tree (ignoring the role of the axon).

\vspace{0.3cm}
\noindent\textit{\bf{The dendritic tree:}}
We model the dendritic tree of each neuron by a uniform infinite one-dimensional cable $\R$, with the soma located a position $0$.
Let $V^i_t(\xi), \xi\in \R$ be the membrane potential at $\xi\in \R$ on the dendritic tree of the $i$th neuron at time $t\geq0$.  We then model the behavior of $V^i_t(\xi)$ by the cable equation
\begin{equation}
\label{cable equation}
\partial_t V_t^i(\xi) = \frac{1}{2} \partial_{\xi}^2 V_t^i(\xi) - \gamma V_t^i(\xi) + f^i_{ext}(t, \xi), \qquad t>0, \xi\in\R,
\end{equation}
for some $ \gamma> 0$, where $f^i_{ext}(t, \xi)$ is the applied current density at time $t$ and position $\xi$ (the applied current per unit length).   This is quite a standard model for the membrane potential across a one-dimensional neuron (see \cite{gerstner:kistler} Section 2.5.1 or \cite{tuckwell}, Chapter 4). 

\vspace{0.3cm}
\noindent\textit{{\bf The soma:}}  This is the processing unit of the neuron that decides whether or not to instigate an action potential.  We use the noisy integrate-and-fire model to describe the behavior.  In this model the electrical potential across the soma evolves continuously according to a stochastic differential equation until it hits a threshold (which for simplicity we take to be $1$), at which point it responds by instigating an action potential.  After action potential initiation the electrical potential across the soma is instantaneously reset to 0.

Precisely, let $U_t^i$ be the potential across the $i$th soma at time $t\geq0$, which evolves according to
\begin{equation}
\label{noisy integrate-and-fire}
U_t^i = U^i_0  + \int_0^t b(U_s^i)ds + I_i(t) - M^i_t + \int_0^t\sigma(U_s^i)dW_s^i, \quad t>0,
\end{equation}
where $U^i_0<1$ almost surely, $(W_t^i)_{t\geq0}$ is a Brownian motion, assumed independent from neuron to neuron, and $I_i(t)$ is the input current into the $i$th soma at time $t$.  Moreover, the instantaneous resets of the process $(U^i_t)_{t\geq0}$ from the threshold $1$ to $0$ are described by the process $(M^i_t)_{t\geq0}$ defined by
\[
M_t^i := \sum_{k=1}^\infty \Ind_{[0,t]}(\tau_k^i), \qquad \tau_k^i = \inf\{t\geq \tau_{k-1}^i: U^i_{t-} \geq 1\}, \quad \tau_k^0 = 0.
\]
Hence $\tau_k^i$ is the $k$th spiking time of neuron $i$, and it is clear that whenever $U^i_{t-} =1$ then $U^i_{t} = 0$.

\vspace{0.3cm}
\noindent\textit{{\bf The coupling:}}
The point is that in order to describe the network effect on the system, equations \eqref{cable equation} and \eqref{noisy integrate-and-fire} should be coupled.  In particular, $f^i_{ext}(t, \xi)$ in \eqref{cable equation} should depend on the spike trains of all the presynaptic neurons connected to neuron $i$, while $I_i(t)$ in \eqref{noisy integrate-and-fire} should depend on the potential across the dendritic tree of neuron $i$.

The latter coupling is easier to describe.  Indeed, we simply impose that
\begin{equation}
\label{form I}
I_i(t) = V_t^i(0), \quad \forall t\geq0,
\end{equation}
which models the fact that the input current to the soma is proportional to the value of the potential across the dendritic tree at position $0$ (where the soma is located).

\begin{rem}
Let us note that up until this point our model is similar to the one proposed in \cite{bressloff:coombes}.  We have however added a noise term, and our purpose is very different.
\end{rem}

For the other coupling, for full generality we would like to take
\begin{equation}
\label{form f wanted}
f^i_{ext}(t, \xi) = \sum_{j:j \neq i} \rho^{(N)}_{j\to i}(\xi) \sum_{k=1}^\infty \delta_0(t - \tau_k^j), \qquad \forall t>0, \xi\in\R,
\end{equation}
where the function $\rho^{(N)}_{j\to i}$ describes the density of synapses on the dendritic tree of neuron $i$ coming from presynaptic neuron $j$, and $\sum_{k} \delta_0(t - \tau_k^j)$ is the spike train of neuron $j$ (recall that $\tau_k^j$ is the time of the $k$th spike of neuron $j$).  Here $\delta_0$ is the Dirac distribution at zero.

However, in a first instance we make the following simplification.  In fact we take
\begin{equation}
\label{form f}
f^i_{ext}(t, \xi) := \frac{1}{S^{N}_i}\sum_{j=1}^N J_{ij}\rho(\xi) \sum_{k=1}^\infty \delta_0(t - \tau_k^j), \qquad \forall t>0, \xi\in\R,
\end{equation}
where $J_{ij}\geq0$ for all $i, j\in\{1, \dots, N\}$, $S^{N}_i := \sum_{j=1}^NJ_{ij}$ and $\rho:\R\to\R$ is smooth and bounded.

\begin{rem}
The interpretation of this simplification is that we are supposing that the shape of the distribution of synapses on the dendritic tree of a neuron is homogeneous throughout the network.  However, we do not suppose that the \text{strengths} of the connections are homogeneous (they are given by $J_{ij}$). 
\end{rem}

\vspace{0.3cm}
\noindent\textit{{\bf Summary:}}
Combining \eqref{cable equation}, \eqref{noisy integrate-and-fire}, \eqref{form I} and \eqref{form f}, we arrive at the following coupled system
 \begin{equation}
\label{coupled particle system}
\left\{\begin{split}
\partial_t V_t^i(\xi) & = \frac{1}{2} \partial_{\xi}^2 V_t^i(\xi) - \gamma V_t^i(\xi) + \frac{\rho(\xi)}{S^N_i} \sum_{j=1}^NJ_{ij} \sum_{k=1}^\infty \delta_0(t - \tau_k^j)\\
U_t^i &= U^i_0  + \int_0^t b(U_s^i)ds + V^i_t(0) - M^i_t + \int_0^t\sigma(U_s^i)dW_s^i,
\end{split}\right.
\end{equation}
for $t\geq0$ and $i\in\{1, \dots, N\}$.   We would like to use this to derive a self-contained system for the potentials  $(U^i)_{i\in\{1,\dots, N\}}$ across each soma in the network, since it is the potential across the soma that determines the behavior of each neuron.  To this end we can in fact solve the first equation in \eqref{coupled particle system} with the help of Green's function, which we define by
\[
\mathcal{G}(t,\xi) := \frac{1}{\sqrt{2\pi t}}e^{-\gamma t}\exp(-\xi^2/2t), \qquad t>0, \xi\in\R.
\]
This is the fundamental solution to the cable equation.
Indeed, then for each $i\in\{1, \dots, N\}$, $t>0$ and  $\xi'\in\R$
\begin{align*}
V_t^i(\xi') = [\mathcal{G}(t, \cdot) *V_0^i](\xi')  + \int_0^t\int_{-\infty}^\infty \frac{\mathcal{G}( t -s, \xi - \xi')}{S^N_i}\rho(\xi) \sum_{j=1}^NJ_{ij} \sum_{k=1}^\infty \delta_0(s - \tau_k^j)d\xi ds,
\end{align*}
where we have used the standard $*$ notation to denote convolution.  Moreover, we have that, in the sense of distributions, $\sum_{k} \delta_0(s - \tau_k^j) = [d/ds]M_s^j$ for all  $s\geq0$, $j\in\{1, \dots, N\}$.
Thus, by evaluating $V_t^i$ at $\xi'=0$, we can write
\begin{align}
\label{before IBP}
V_t^i(0) &= [\mathcal{G}(t, \cdot ) *V_0^i](0) + \int_0^t[\mathcal{G}(t -s,\cdot)*\rho](0)  \frac{d}{ds}\left(\frac{1}{S^N_i} \sum_{j=1}^N J_{ij}M_s^j\right) ds.
\end{align}

\vspace{0.3cm}
\noindent\textit{{\bf Assumption of no synapses on the soma:}}  Since the integral on the right-hand side of \eqref{before IBP} involves a smooth function multiplied by a distribution, we must understand it with the help of an integration by parts.  In order to avoid boundary terms (which would significantly complicate matters since then $t\mapsto V_t^i(0)$ would be discontinuous -- see \cite{DIRT, DIRT2}), we make the biologically plausible assumption that there are no synapses on the soma itself.  Mathematically we translate this into the assumption that $\rho: \R\to\R$ is such that
\begin{equation}
\label{assumption rho}
\rho(0) = \rho''(0) = \rho^{(iv)}(0) = 0.
\end{equation}
Such an assumption is not uncommon in the biological literature (see for example \cite{bressloff:coombes}).  The intuition is that there will be no instantaneous transmission of a presynaptic action potential to the soma, and that the input current to $i$th the soma, which is given by $V_t^i(0)$, is continuous in time.

Anyway, under this assumption, after an integration by parts, \eqref{before IBP} becomes
\begin{equation}
\label{V particle}
\begin{split}
V_t^i(0) &= [\mathcal{G}(t, \cdot) *V_0^i](0) + \sum_{j=1}^N\frac{J_{ij}}{S^N_i} \int_0^t\frac{d}{dt}[\mathcal{G}(t -s, \cdot)*\rho](0) M_s^j ds.
\end{split}
\end{equation}
By substituting this into the second equation in \eqref{coupled particle system}, we derive the desired self-contained system for $(U^i)_{i\in\{1,\dots, N\}}$:
\begin{equation}
\label{U particle self contained}
U_t^i = U^i_0 + H_{V_0^i}(t)  + \int_0^t b(U_s^i)ds + \sum_{j=1}^N\frac{J_{ij}}{S^N_i} \int_0^t{G}_\rho(t -s) M_s^j ds - M^i_t + \int_0^t\sigma(U_s^i)dW_s^i,
\end{equation}
where we have used the notation ${G}_\rho(t -s):= (d/dt)[\mathcal{G}(t -s, \cdot)*\rho](0)$ and $H_{V_0^i}(t) := [\mathcal{G}(t, \cdot) *V_0^i](0)$.
We thus see that \eqref{U particle self contained} is exactly of the form \eqref{particle system} with $H=H_{V_0^i}$ and $G=G_\rho$. The behavior of such interacting particle systems in the limit as $N\to\infty$ is of great interest, since there are typically very large numbers of neurons in neural networks.

We must finally check that $H_{V_0^i}$ and ${G}_\rho$ satisfy the conditions stated in Assumptions \ref{assumptions}, i.e. that $G_\rho$ and $H_{V_0^i}$ are twice continuously differentiable and bounded with bounded derivatives, and that $G_\rho(0) = G_\rho'(0) = 0$.
This is clearly true for $H_{V_0^i}$ whenever $V^i_0$ is bounded on $\R$.  For ${G}_\rho$, note that by definition
\begin{align*}
G_\rho(t) 
= \frac{d}{dt}\left(\frac{e^{-\gamma t}}{\sqrt{2\pi}}\int_{-\infty}^\infty e^{-\frac{\xi^2}{2}}\rho(\xi\sqrt{t})d\xi\right)
= \frac{e^{-\gamma t}}{\sqrt{2\pi}}\int_{-\infty}^\infty e^{-\frac{\xi^2}{2}}\left[ \rho''(\xi\sqrt{t}) -\gamma\rho(\xi\sqrt{t})\right]d\xi
\end{align*}
so that $G_\rho(0) = 0$ by \eqref{assumption rho}, and  $\|G_\rho\|_\infty <\infty$ since  $\|\rho\|_{\mathcal{C}_b^2(\R)}<\infty$.  Moreover

\begin{align*}
&\frac{d}{dt}G_\rho(t) 
= \frac{e^{-\gamma t}}{\sqrt{2\pi}}\int_{-\infty}^\infty e^{-\frac{\xi^2}{2}}\left[ \gamma^2\rho(\xi\sqrt{t}) - 2 \gamma\rho''(\xi\sqrt{t}) +\rho^{(iv)}(\xi\sqrt{t})\right]d\xi,
\end{align*}
so that $G'_\rho(0)=0$ by \eqref{assumption rho} again, and $G_\rho'$ is bounded.  By taking another derivative, one can also see that the second derivative of $G_\rho$ exists and is continuous and bounded.

\section{Proof of Theorem \ref{thm: existence and uniqueness}: Well-posedness of limit equation \eqref{limit equation}}
\label{sec: e+u}

The aim of this section is to prove Theorem \ref{thm: existence and uniqueness}, i.e the existence and uniqueness of a solution to the nonlinear equation \eqref{limit equation}.  Due to the discontinuous nature of any solution to \eqref{limit equation} (the potential $U$ jumps to $0$ whenever it reaches $1$), it is convenient to reformulate the system as a continuous one, by instead looking at the dynamics of $Z_t := U_t+M_t$.  By \eqref{limit equation}, it is straightforward to check that $Z$ must then satisfy
\begin{equation}
\begin{cases}
Z_t &= U_0 + H(t)  + \int_0^t b(Z_s - M_s)ds +\int_0^t   {G}(t -s)\E(M_s) ds + \int_0^t\sigma(Z_s - M_s)dW_s\\
M_t &= \lfloor(\sup_{s\leq t} Z_s )_{+}\rfloor,\\
\tau_k &= \inf\{t\geq0: Z_t \geq k\},\ k\in\mathbb{N}\backslash\{0\}, \ \tau_0 = 0,
\end{cases}
\label{Z limit equation}
\end{equation}
where $\lfloor x \rfloor$ and $(x)_+$ indicate the integer part of $x$ and $\max\{x, 0\}$ respectively, for any $x\in\R$. 
Indeed, since $U_{t}<1$ it is clear that for a given $k \geq 0$ such that 
$\tau_{k} \leq t < \tau_{k+1}$,  
\[
 \lfloor(\sup_{s\leq t} Z_s )_{+}\rfloor \leq  \lfloor(\sup_{s\leq t} U_s )_{+}\rfloor + k = k =M_{t}.
\]
\rev{Conversely, for such a $t$, $M_t = M_{\tau_k} = k = Z_{\tau_k} \leq \sup_{s\leq t} \lfloor (Z_s)_+\rfloor = \lfloor (\sup_{s\leq t}Z_s)_+\rfloor$}
so that the second equality in \eqref{Z limit equation} holds.  
The point is that any solution $(Z_t)_{t\geq0}$ to \eqref{Z limit equation} now has continuous paths, and moreover a unique solution to  \eqref{Z limit equation} will yield a unique solution to the original limit equation \eqref{limit equation} (one can recover a solution to \eqref{limit equation} by setting $U_t = Z_t - M_t$).

As is usual in the study of McKean-Vlasov equations, to prove the existence and uniqueness of a solution to \eqref{Z limit equation} we look for a fixed point of an associated map.  

\rev{To this end, let $T>0$ be arbitrary and for a function $h\in\mathcal{C}^1([0, T])$ set
\begin{equation}
\label{fh}
f_h(t) := H(t)  +\int_0^t   {G}(t -s)h(s)ds, \qquad t\geq0,
\end{equation}
(which is continuous and twice differentiable by Assumptions \ref{assumptions}). Define the process $(Z^h_t)_{t\in[0,\tau^h_1\wedge T]}$ by
\[
Z^h_t = U_0 + f_h(t)  + \int_0^t b(Z^h_s)ds + \int_0^t\sigma(Z^h_s)dW_s, \quad t\in [0, \tau^h_1\wedge T],
\]
where $\tau_1^h := \inf\{t\in[0, T]: Z_t^h \geq 1\}$ ($\inf\{\emptyset\} = \infty$).  This process is well-defined by standard results, since the coefficients are Lipschitz, and $\tau_1^h >0$ since $U_0<1$.  For any $k\geq1$, if $\tau^h_k <T$, by recurrence define $(Z^h_t)_{t\in(\tau^h_k,\tau^h_{k+1}\wedge T]}$ by 
\[
Z^h_t =  Z^h_{\tau^h_k}+ f_h(t)  + \int_{\tau^h_k}^t b(Z^h_s - k)ds + \int_{\tau^h_k}^t\sigma(Z^h_s -k)dW_s, \quad t\in(\tau^h_k,\tau^h_{k+1}\wedge T],
\]
where $\tau_{k+1}^h := \inf\{t\in[0, T]: Z_t^h \geq k+1\}$ ($\inf\{\emptyset\} = \infty$), which is again well-defined by classical results. In this way we can well-define the process $(Z^h_t)_{t\in[0,\tau^h_\infty\wedge T]}$ where $\tau^h_\infty = \lim_{k\to\infty}\tau^h_k$, which satisfies the equation 
\begin{equation}
\label{Zh}
Z^h_t = U_0 + f_h(t)  + \int_0^t b(Z^h_s - M^h_s)ds + \int_0^t\sigma(Z^h_s - M^h_s)dW_s,
\end{equation}
where $M^h_t := \lfloor(\sup_{s\leq t} Z^h_s )_{+}\rfloor$. 

Note that by definition $\tau^h_\infty$ is the first explosion time of this process.  We now show that $\tau^h_\infty =\infty$ almost surely. Indeed, by the linear growth condition on $b$ and the fact that $|U_0| \leq R$ a.s. (see Assumptions \ref{assumptions}), we have that for all $t\in[0, \tau^h_\infty\wedge T]$
\begin{align}
\label{Lipschitz calc}
\sup_{s\leq t} |Z^h_s| \leq R + \|f_h\|_{\infty,t} + \Lambda_b t +2\Lambda_b\int_0^t\sup_{r\leq s}|Z^h_r|ds + \sup_{s\leq t}\left|\int_0^s \sigma(Z^h_r - M^h_r)dW_r\right|,
\end{align}
where we have also used the fact that $M^h_t \leq \sup_{s\leq t} |Z^h_s|$. 
By the Burkh\"older-Davis-Gundy inequality, the boundedness of $\sigma$ and Gronwall's lemma this yields the fact that $\E(\sup_{s\leq T\wedge\tau^h_\infty} |Z^h_s|)$ is finite. In particular, this implies that $\tau^h_\infty = \infty$ almost surely (if $\tau^h_\infty\leq T$ we would have a contradiction). We can thus conclude that $(Z^h_t)_{t\in[0,T]}$ is well-defined.
Finally we set
\begin{equation}
\label{Phi}
\Phi(h)(t) := \E(M_t^h), \quad t\in [0, T],
\end{equation}
which is well-defined by the above arguments (using again the fact that $M^h_t \leq \sup_{s\leq t} |Z^h_s|$).
}


\subsection{The fixed points of $\Phi$ correspond to the solutions of \eqref{Z limit equation}}
In this subsection, we prove that finding a unique solution to \eqref{Z limit equation} is equivalent to finding a unique fixed point of $\Phi$.  This is the object of the following lemma.  It should be noted that in the proof of this lemma we use the assumption on the initial condition $U_0$ in Assumptions \ref{assumptions}.

\begin{lem}
\label{solution fixed point}
Suppose $h\in\mathcal{C}^1([0, T])$ is a fixed point of $\Phi$.  Then a solution $(Z^h_t)_{t\in[0, T]}$ to \eqref{Zh} is a solution to the nonlinear equation \eqref{Z limit equation}.  Conversely, if $(Z_t)_{t\in[0, T]}$ is a solution to the nonlinear equation \eqref{Z limit equation}, then $h(t) = \E(M_t)$, $t\in[0, T]$ is a fixed point of $\Phi$ in $\mathcal{C}^1([0, T])$.
\end{lem}

Before giving the proof of this result, we first establish a useful representation of $\Phi$ in terms of the cumulative distribution functions of first hitting times.  Indeed, note that we can write
\[
\Phi(h)(t) = \sum_{k\geq1} \mathbb{P}(\tau^h_k \leq t), \quad t\in[0, T].
\]  
We use the notation $\mathbb{P}_0$ to indicate the conditional probability given the process starts at $0$, and define 
\begin{equation}
\label{fshift}
f^{\sharp s}_h(r):=  f_h(r+s) - f_h(s), \quad 0\leq s \leq t \leq T,\ r\in [0, t-s],
\end{equation}
where $f_h$ is given by \eqref{fh}. Then, by the strong Markov property, we see that
\begin{align}
\label{markov prop}
\Phi(h)(t) &=  \mathbb{P}(\tau^h_1 \leq t) +  \sum_{k\geq2} \int_0^t\mathbb{P}(\tau^{h}_k \leq t | \tau^{h}_{k-1} \in ds)\mathbb{P}(\tau^h_{k-1}\in ds)\nonumber\\
&=  \mathbb{P}(\tau^h_1 \leq t) +  \sum_{k\geq2} \int_0^t\mathbb{P}_0(\tau^{h,\sharp s}_1 \leq t-s )\mathbb{P}(\tau^h_{k-1}\in ds),
\end{align}
where for $s\leq t$, $\tau^{h,\sharp s}_1$ is the first time the process $(X^{h, \sharp s}_{r})_{r\in [0, t-s]}$ reaches the level $1$, and $(X^{h, \sharp s}_{r})_{r\in [0, t-s]}$ is given by
\begin{equation}
\label{shifted process}
X^{h, \sharp s}_{r} = X^{h, \sharp s}_{0} + f^{\sharp s}_h(r) + \int_0^rb(X^{h, \sharp s}_{u})du + \int_0^r\sigma(X^{h, \sharp s}_{u})dW_u,
\end{equation}
for $r\in[0, t-s]$.

The usefulness of the representation \eqref{markov prop} is that we have some nice bounds on the densities of the hitting times of the threshold $1$ for general non-homogeneous processes of the form \eqref{shifted process}, which are detailed in the Appendix.  It should be noted that the bounds in the Appendix are obtained using purely probabilistic arguments (in contrast to the related small time bounds used in \cite{DIRT} and proved in \cite{DIRT3}).

\begin{proof}[Proof of Lemma \ref{solution fixed point}]
The first claim is straightforward.  For the second suppose that  $(Z_t)_{t\in[0, T]}$ is a solution to \eqref{Z limit equation}, and let $h(t) = \E(M_t)$. Then, by a standard pathwise uniqueness argument, $(Z_t)_{t\in[0, T]} = (Z^h_t)_{t\in[0, T]}$, so that
\[
\Phi(h)(t) = \E(M^h_t) = \E(M_t) = h(t), \quad t\in[0, T].
\]
To complete the proof, we have to check that $h\in\mathcal{C}^1([0, T])$, i.e. that $h$ is in the domain of $\Phi$.  To this end, first note that since $\Phi(h) = h$, it suffices to show that $\Phi(h)$ is continuously differentiable.  By the representation \eqref{markov prop}, $\Phi(h)$ is clearly differentiable if both $\mathbb{P}(\tau^h_1 \leq t)$ and $\mathbb{P}_0(\tau^{h,\sharp s}_1 \leq t-s)$ are continuously differentiable in $t$.  For the latter, $\tau^{h,\sharp s}_1$ is the first time that $X^{h, \sharp s}$ given by \eqref{shifted process} and started at $0$ reaches $1$. In addition  the non-homogeneous function $f^{\sharp s}_h$ driving $X^{h, \sharp s}$ is $\mathcal{C}^2$ by definition \eqref{fh}, since
\begin{equation}
\label{fC2}
\frac{d^2}{dr^2}f^{\sharp s}_h(r) = f''_h(r+s) = H''(r+s) + \int_0^{r+s}\frac{d^2}{dr^2}G(r+s-u)h(u)du, \quad r\in[0, t-s],
\end{equation}
where we have used the fact that $G(0) = G'(0) = 0$.  This is finite for all $r\in[t-s]$ since $h$ is non-decreasing (and measurable), together with the assumptions stated in Assumptions \ref{assumptions} on $H$ and $G$.  Thus the equation for $X^{h, \sharp s}$ is of the form \eqref{SDE} studied in the Appendix, so that we can apply Proposition \ref{Pauwels} to see that 
\[
[d/dt]\mathbb{P}_0(\tau^{h,\sharp s}_1 \leq t-s) = p^0_{f_h^{\sharp s}}(t-s),
\] 
using the notation of the Appendix, which exists, is continuous in $t$, and is bounded thanks to Proposition \ref{density bound absolute}.

For the former term, note that
\[
\mathbb{P}(\tau^h_1 \leq t) 
= \int_{-\infty}^1\mathbb{P}_x(\tau^{h, \sharp 0}_1 \leq t)\mathbb{P}(U_0\in dx).
\]
By Proposition \ref{density bound absolute} of the Appendix (which is again applicable since  the non-homogeneous function $f^{\sharp 0}_h$ driving $X^{h, \sharp 0}$ is $\mathcal{C}^2$ exactly as above) we deduce that there exists a constant $C_T$ (depending on the bounds stated in Assumptions \ref{assumptions}, and in particular on $R$) such that
\begin{align*}
\frac{d}{dt}\mathbb{P}(\tau^h_1 \leq t) &\leq C_T
\int_{-\infty}^1(1-x)t^{-\frac{3}{2}}e^{-\frac{(1-x)^2}{\Lambda_\sigma t}}\mathbb{P}(U_0\in dx).
\end{align*}
By condition \eqref{initial condition} of Assumptions \ref{assumptions}, this is uniformly bounded for $t\in (0,T]$.
\end{proof}

It is worth noting that exactly the same argument used at the end of the above proof yields the following corollary.
\begin{cor}
\label{diffble}
The map $\Phi$ is such that $\Phi:\mathcal{C}^1([0, T])\to\mathcal{C}^1([0, T])$.
\end{cor}


\subsection{Proof of Theorem \ref{thm: existence and uniqueness}: Uniqueness.}
The purpose of this section is to prove the uniqueness statement in Theorem \ref{thm: existence and uniqueness}.
The proof is based on the following key lemma, which holds thanks to the regularizing effect of the kernel $G$. In contrast to many nonlinear problems (see in particular \cite{DIRT}), the presence of $G$ means that we do not need to split the proof of uniqueness into many small time steps.

\begin{lem}
\label{integral lem}
Let $T>0$ and $h, \tilde{h} \in \mathcal{C}^1([0, T])$.  Then there exists a constant $C_T$, depending only on $T$ and $A_T=\max\{\|h\|_{\infty, T}, \|\tilde{h}\|_{\infty, T}, \Phi(\tilde{h})(T)\}$ (as well the bounds stated in Assumptions \ref{assumptions}), such that
\[
\|\Phi(h) - \Phi(\tilde{h})\|_{\mathcal{C}^1([0, T])} \leq C_T\int_0^T\| h - \tilde{h}\|_{ \mathcal{C}^1([0, s])} ds.
\]
In particular, $\Phi: \mathcal{C}^1([0, T]) \to   \mathcal{C}^1([0, T])$ is continuous.
\end{lem}

We postpone the proof of Lemma \ref{integral lem} to Section \ref{sec: proof of integral lem}, and here proceed instead to the proof of uniqueness in Theorem \ref{thm: existence and uniqueness}.
As remarked above, it is equivalent to prove the uniqueness of a solution to \eqref{Z limit equation} on any interval $[0, T]$.  To this end, suppose $(Z_t, M_t)_{t\in[0, T]}$ and $(\tilde{Z}_t, \tilde{M}_t)_{t\in[0, T]}$ are two solutions to \eqref{Z limit equation}.  By Lemma \ref{solution fixed point} we have that $h(t)=\E(M_t)$ and $\tilde{h}(t)=\E(\tilde{M}_t)$ are fixed points of $\Phi$ in $\mathcal{C}^1([0, T])$.  Moreover, by that result $(Z, M) = (Z^h, M^h)$, where $(Z^h, M^h)$ is the solution to \eqref{Zh}.  Similarly $(\tilde{Z}, \tilde{M}) = (Z^{\tilde{h}}, M^{\tilde{h}})$.  

We then apply Lemma \ref{integral lem} to $h$ and $\tilde{h}$.  This yields
\[
\|\Phi(h) - \Phi(\tilde{h})\|_{\mathcal{C}^1([0, T])} \leq C_T\int_0^T\|h - \tilde{h}\|_{\mathcal{C}^1([0, r])}dr,
\]
for some constant $C_T$ depending on $T$ and $\max\{\|h\|_{\infty, T}, \|\tilde{h} \|_{\infty, T}, \Phi(\tilde{h})(T)\}$ (as well the bounds stated in Assumptions \ref{assumptions}).  However, since $h$ and $\tilde{h}$ are fixed points of $\Phi$, this yields
\[
\|h - \tilde{h}\|_{\mathcal{C}^1([0, T])} \leq C_T\int_0^T\|h - \tilde{h}\|_{\mathcal{C}^1([0, r])}dr.
\]
By Gronwall's lemma we conclude that $h\equiv \tilde{h}$.  By the pathwise uniqueness of a solution to the SDE \eqref{Zh}, we have $(Z^h, M^h) = (Z^{\tilde{h}}, M^{\tilde{h}})$ almost surely.  Hence $(Z, M) = (\tilde{Z}, \tilde{M})$ almost surely.
\qed

\subsection{Proof of Lemma \ref{integral lem}}
\label{sec: proof of integral lem}
Let $t\in[0, T]$ and  $h, \tilde{h} \in \mathcal{C}^1([0, T])$.  Then, using the representation \eqref{markov prop}, we can see that
\begin{equation}
\label{phi formula}
\Phi(h)(t) = \mathbb{P}(\tau^h_1 \leq t) +  \int_0^t\mathbb{P}_0(\tau^{h,\sharp s}_1 \leq t-s)\Phi(h)'(s)ds,
\end{equation}
where we recall that $\tau^{h,\sharp s}_1$ is the first time $X^{h, \sharp s}$ reaches the level $1$, and $X^{h, \sharp s}$ is given by \eqref{shifted process}. We note that \eqref{phi formula} is valid since $\Phi(h)$ is continuously differentiable by Corollary \ref{diffble}.  By differentiating \eqref{phi formula} with respect to $t$, we see that
\begin{align}
\label{main}
|\Phi(h)'(t) - \Phi(\tilde{h})'(t)| &\leq \left|\frac{d}{dt}\mathbb{P}(\tau^h_1 \leq t) - \frac{d}{dt}\mathbb{P}(\tau^{\tilde{h}}_1 \leq t) \right| \nonumber\\
& \quad + \int_0^t\left|\frac{d}{dt}\mathbb{P}_0(\tau^{h,\sharp s}_1 \leq t-s)\right|\left|\Phi(h)'(s) - \Phi(\tilde{h})'(s)\right|ds\nonumber\\
&\quad +  \int_0^t\left|\frac{d}{dt}\mathbb{P}_0(\tau^{h,\sharp s}_1 \leq t-s)-\frac{d}{dt} \mathbb{P}_0(\tau^{\tilde{h},\sharp s}_1 \leq t-s)\right|\Phi(\tilde{h})'(s)ds, 
\end{align}
where we have used the fact that $\Phi(\tilde{h})(s)$ is non-decreasing.  The first thing to note is that by Proposition \ref{density bound absolute} of the Appendix, there exists a constant $C_T$ that only depends on $T, A_T$ and the bounds of Assumptions \ref{assumptions} such that
\begin{equation}
\label{Q abs bound}
\left|\frac{d}{dt}\mathbb{P}_0(\tau^{h,\sharp s}_1 \leq t-s)\right| \leq C_T, \quad \forall\ 0\leq s\leq t \leq T.
\end{equation}
This follows because, exactly as noted in the proof of Lemma \ref{solution fixed point}, $\tau^{h,\sharp s}_1$ is the first time that $X^{h, \sharp s}$ reaches $1$, and the non-homogeneous drift $f^{\sharp s}_h$ driving $X^{h, \sharp s}$ is bounded in the $\mathcal{C}^2([0, T-s])$ norm by a constant depending only on $T$ and the bounds of Assumptions \ref{assumptions}.  This can be seen directly from \eqref{fC2}.
Thus, in the notation introduced in the Appendix,
\[
\frac{d}{dt}\mathbb{P}_0(\tau^{h,\sharp s}_1 \leq t-s) = p^0_{f^{\sharp s}_h}(t-s), \quad \forall\ 0\leq s\leq t \leq T,
\]
which is indeed bounded by Proposition \ref{density bound absolute}.

Using \eqref{Q abs bound} in \eqref{main}
we see that
\begin{align}
\label{main2}
|\Phi(h)'(t) - \Phi(\tilde{h})'(t)| &\leq  \left|\frac{d}{dt}\mathbb{P}(\tau^h_1 \leq t) - \frac{d}{dt}\mathbb{P}(\tau^{\tilde{h}}_1 \leq t) \right| + C_T\int_0^t\left|\Phi(h)'(s) - \Phi(\tilde{h})'(s)\right|ds\nonumber\\
&\quad +  C_T\sup_{s\leq t}\left|\frac{d}{dt}\mathbb{P}_0(\tau^{h,\sharp s}_1 \leq t-s)- \frac{d}{dt}\mathbb{P}_0(\tau^{\tilde{h},\sharp s}_1 \leq t-s)\right|
\end{align}
for some $C_T$ that again depends only on $T, A_T$ and the bounds in Assumptions \ref{assumptions}.  Note that it is at this point that the constant $C_T$ may depend through $A_T$ on $\Phi(\tilde{h})(T)$.  The above dependencies  will be true for all constants $C_T$ below in the proof, though it will be allowed to increase from line to line.
By introducing the notation
\[
D^{h,\tilde{h}}_{s, t}(x) := \left|\frac{d}{dt}\mathbb{P}_x(\tau^{h,\sharp s}_1 \leq t-s)-\frac{d}{dt} \mathbb{P}_x(\tau^{\tilde{h},\sharp s}_1 \leq t-s)\right|,
\]
for $0\leq s\leq t \leq T$ and $x<1$, we continue \eqref{main2} (by conditioning on the initial condition in the first term) as
\begin{align}
\label{main3}
|\Phi(h)'(t) - \Phi(\tilde{h})'(t)| &\leq  \int_{-\infty}^1 D^{h,\tilde{h}}_{0, t}(x)\mathbb{P}(U_0 \in dx) + C_T\int_0^t\left|\Phi(h)'(s) - \Phi(\tilde{h})'(s)\right|ds\nonumber\\
&\qquad \qquad + C_T\sup_{s\leq t}D^{h,\tilde{h}}_{s, t}(0).
\end{align}
The point is that we can use Proposition \ref{density bound diff} of the Appendix (which is applicable for the same reasons discussed above) to deduce that
\begin{align*}
D^{h, \tilde{h}}_{s, t}(x) &\leq C_Te^{C_Tx^2}\|f^{\sharp s}_h - f^{\sharp s}_{\tilde{h}}\|_{\mathcal{C}^2([0, t-s])}\frac{1 -x}{\sqrt{2\pi (t-s)^3}}e^{-\frac{(1-x)^2}{\Lambda_\sigma(t-s)}}\\
&\leq  C_Te^{C_Tx^2}\int_0^t|h(r) - \tilde{h}(r)|dr\frac{1 -x}{\sqrt{2\pi (t-s)^3}}e^{-\frac{(1-x)^2}{\Lambda_\sigma(t-s)}},
\end{align*}
for all $0\leq s\leq t\leq T$ and $x<1$, where we have used the definition of $f^{\sharp s}_h$, $f^{\sharp s}_{\tilde{h}}$ and the boundedness of $G$ to pass from the first to the second line.
Using this in \eqref{main3} then yields (thanks to the assumptions on $U_0$)
\begin{align*}
|\Phi(h)'(t) - \Phi(\tilde{h})'(t)| &\leq C_T\int_0^t\left|\Phi(h)'(s) - \Phi(\tilde{h})'(s)\right|ds+ C_T\int_0^t|h(r) - \tilde{h}(r)|dr,
\end{align*}
for all  $t\in[0, T]$.  An application of Gronwall's lemma yields
\begin{align}
\label{main5}
|\Phi(h)'(t) - \Phi(\tilde{h})'(t)| &\leq C_T\int_0^t|h(r) - \tilde{h}(r)|dr, \quad t\in[0, T].
\end{align}
To complete the proof we finally remark that this implies
\begin{align*}
\|\Phi(h) - \Phi(\tilde{h})\|_{\mathcal{C}^1([0, T])} &\leq \rev{C_T}\|\Phi(h)' - \Phi(\tilde{h})'\|_{\infty, T} \\
& \leq C_T\int_0^T\|h - \tilde{h}\|_{\infty, r}dr \leq C_T\int_0^T\|h - \tilde{h}\|_{\mathcal{C}^1([0, r])}dr,
\end{align*}
where we have used the fact that $\Phi(h)(0)=\Phi(\tilde{h})(0) =0$.
\qed


\subsection{Proof of Theorem \ref{thm: existence and uniqueness}: Existence.}
The purpose of this section is to prove the existence statement in Theorem \ref{thm: existence and uniqueness}.  

\rev{
\begin{rem}
It has been pointed out to us by the referee that in fact this section could be skipped.  This is because in the course of the proof of Theorem \ref{thm: convergence} below (see Section \ref{sec: proof of convergence}) we prove the existence of a solution to the nonlinear martingale problem associated to \eqref{Z limit equation} by studying the convergence of the approximating particle system.  This yields the existence of a weak solution to \eqref{Z limit equation}.  Thanks to the fact that we already have strong uniqueness by the preceding section, the details of the classical argument of Yamada and Watanabe (see \cite[Proposition 5.3.20]{KS})
can be checked to deduce strong existence.

However, we choose to include this section, since the technique is quite simple and is directly related to the nonliner SDE \eqref{Z limit equation}.  In particular it allows us to complete the proof of Theorem \ref{thm: existence and uniqueness} without resorting to the approximating particle system.
 \end{rem}
}

We use an iterative scheme. Indeed, fix $T>0$, define $h_0 \equiv 0$, and let
\begin{equation}
\label{hn}
h_n(t) :=\Phi(h_{n-1})(t), \quad t\in[0,T], \quad n\geq1.
\end{equation}
The idea is to apply Lemma \ref{integral lem} repeatedly to show that $(h_n)_{n\geq1}$ forms a convergent sequence in $\mathcal{C}^1([0, T])$.  In order to do this we need the following stability result.


\begin{lem}
\label{stable subspace lem}
There exists a non-decreasing function $g:[0, \infty) \to [0, \infty)$ such that if $h \in \mathcal{C}^1([0, T])$ is non-decreasing and such that $h(0)=0$ and $h(t) \leq g(t)$ for all $t\in[0, T]$, then $\Phi(h)(t) \leq g(t)$, for all $t\in[0, T]$.
\end{lem}

We again postpone the proof of this technical lemma to the next section, and first show how it can be used to deduce the existence statement in Theorem \ref{thm: existence and uniqueness}. By definition, we have
\[
h_n = \underbrace{\Phi\circ\dots\circ\Phi}_{n\ \mathrm{times}}(h_0).
\]
In particular, $h_n$ is well-defined in $\mathcal{C}^1([0, T])$ $\forall\ n\geq0$ thanks to Corollary \ref{diffble}. Moreover (by the definition of $\Phi$ in \eqref{Phi}) $h_n$ is non-decreasing and $h_n(0)=0$ for all $n\geq0$.

Let $g:[0, \infty)\to[0,\infty)$ be as in Lemma \ref{stable subspace lem}. Since $h_0(t) \leq g(t)$ $\forall t\in[0, T]$ (recall $h_0\equiv 0$ by definition), by Lemma  \ref{stable subspace lem} it holds that $h_1(t) = \Phi(h_0)(t) \leq g(t)$ $\forall t\in[0, T]$. By repeating this argument, we see that for any $n\geq0$
\[
h_n(t) = \Phi(h_{n-1})(t) \leq g(t),\qquad t\in[0, T],
\]
where we emphasize that the right-hand side is independent of $n$.
Thus by Lemma \ref{integral lem}, we have for any $n\geq 1$
\begin{align}
\label{n=1}
\|h_{n+1} - h_{n}\|_{\mathcal{C}^1([0, T])} 
\leq C_T\int_0^T\|h_{n} - h_{n-1}\|_{\mathcal{C}^1([0, r])}dr,
\end{align}
where the constant $C_T$ is also independent of $n$ (it depends only on $T$ and $g(T)$, as well the on the bounds stated in Assumptions \ref{assumptions}).  By iterating inequality \eqref{n=1} $n$ times, we then see that
\[
\|h_{n+1} - h_{n}\|_{\mathcal{C}^1([0, T])} \leq \frac{C^n_TT^n}{n!}\|h_1\|_{\mathcal{C}^1([0, T])}.
\]
Hence $(h_n)_{n\geq1}$ is a Cauchy sequence in $\mathcal{C}^1([0, T])$.  Therefore it has a limit in $\mathcal{C}^1([0, T])$, which we denote by $h_\infty$.  By Lemma \ref{integral lem}, we see that
\[
\Phi(h_\infty) = \lim_{n\to\infty}\Phi(h_n) = \lim_{n\to\infty}h_{n+1} = h_\infty,
\]
so that $h_\infty$ is a fixed point of $\Phi$.  By Lemma \ref{solution fixed point}, this yields a solution to \eqref{Z limit equation}, and hence to \eqref{limit equation} as required.
\qed

\subsection{Proof of Lemma \ref{stable subspace lem}}
We will in fact show that there exists a non-decreasing function $g:[0, \infty) \to [0, \infty)$ such that if $h \in \mathcal{C}^1([0, T])$ is such that $h(0)=0$ and $h$ is non-decreasing then
\begin{equation}
\label{sup bound}
h(t) \leq g(t), \forall t\in[0, T]  \ \Rightarrow\ \E\Big(\sup_{s\in [0,t]}|Z^h_s|\Big) \leq g(t), \forall t\in[0, T],
\end{equation}
where $(Z^h_t)_{t\in[0, T]}$ is given by \eqref{Zh}.
This clearly implies the result, since $\Phi(h)(t) \leq \E(\sup_{s\in[0, t]}|Z^h_s|)$ for all $t\in[0, T]$. The proof will be in two steps.

\vspace{0.2cm}
\noindent\textit{Step 1: Small time.}
We will first show \eqref{sup bound} holds for $T=T_0$ for some small $T_0$  and $g = g_0$ to be chosen below. 
To this end, 
let $(Z_{t}^{h})_{t\in[0, T_0]}$ be the solution to \eqref{Zh} on $[0, T_0]$, and $F(t) :=  \E (\sup_{s \in[0, t]}|Z_{s}^{h}|)$. By \eqref{Lipschitz calc}, we have that
\begin{align}
\label{F1}
F(t)& \leq R + \|f_h\|_{\infty,t} + \Lambda_b t+ c\Lambda_\sigma\sqrt{t} +  2\Lambda_b\int_0^tF(s)ds, \quad t\in[0, T_0].
\end{align}
By definition of $f_h$ in \eqref{fh},
\begin{align}
 \|f_h\|_{\infty,t} 
&\leq \|H\|_\infty + t\|G\|_\infty h(t) \leq \|H\|_\infty+ \frac{1}{2}h(t), \nonumber
\end{align}
for all $t\in [0, T_0]$, where we have now chosen $T_0 =  (2\|G\|_\infty)^{-1}$. Hence, if we assume that $h(t)\leq g_0(t)$ for $t\in[0, T_0]$, it follows from \eqref{F1} that
\begin{align}
\label{Feq}
F(t)&\leq R +  \Lambda_bt + c\Lambda_\sigma\sqrt{t} + 2\Lambda_b\int_{0}^{t} F(s)ds  + \|H\|_\infty+\frac{1}{2}g_0(t), \quad t\in[0, T_0].
\end{align}

The point is that we have assumed  $T_0$ to be small enough so that the coefficient of $g_0(t)$ in the above is $<1$, which allows us to apply a Gronwall type argument.
Indeed, defining $Q(s) := \exp(-2\Lambda_b s)\int_0^s2\Lambda_b F( r)dr$ for $s\in [0, T_0]$ yields
\begin{align*}
Q'(s) &=\left(F(s) - 2\Lambda_b\int_0^s F( r)dr\right)2\Lambda_b e^{-2\Lambda_b s}\\
&\leq \left(R +  \Lambda_bs + c\Lambda_\sigma\sqrt{s} + \frac{1}{2}g_0( s)+  \|H\|_\infty\right)2\Lambda_b e^{-2\Lambda_b s},
\end{align*}
by \eqref{Feq}.  This implies
\begin{align}
\label{Feq2}
2\Lambda_b\int_0^t F( s)ds &\leq e^{2\Lambda_b t}\int_0^t\left(R +  \Lambda_bs + c\Lambda_\sigma\sqrt{s} + \frac{1}{2}g_0( s)+  \|H\|_\infty\right)2\Lambda_b e^{-2\Lambda_b s}ds\nonumber\\
&\leq \left(R +  \Lambda_bt + c\Lambda_\sigma\sqrt{t} +  \|H\|_\infty\right)(e^{2\Lambda_b t} - 1) + \frac{1}{2}e^{2\Lambda_b t}\int_0^tg_0( s)2\Lambda_b e^{-2\Lambda_b s}ds.
\end{align}
Define now $g_0$ by
\begin{equation}
\label{definition g_0}
g_0(t) := 2(R + \Lambda_bt+ c\Lambda_\sigma\sqrt{t} + \|H\|_\infty)e^{4\Lambda_b t}, \qquad t\in[0, T_0].
\end{equation}
Then, we have that
\begin{align*}
\frac{1}{2}e^{2\Lambda_b t}\int_0^tg_0( s)2\Lambda_b e^{-2\Lambda_b s}ds 
&\leq (R + \Lambda_bt + c\Lambda_\sigma\sqrt{t} + \|H\|_\infty)\left(e^{4\Lambda_b t} - e^{2\Lambda_b t}\right),
\end{align*}
so that by \eqref{Feq2}
\begin{align*}
2\Lambda_b\int_0^t F( s)ds &\leq  (R + \Lambda_bt + c\Lambda_\sigma\sqrt{t} +  \|H\|_\infty)\left(e^{4\Lambda_b t} - 1\right).
\end{align*}
Finally, using this in \eqref{Feq}, we have thus shown that 
\begin{align*}
F(t)&\leq (R +\Lambda_bt + c\Lambda_\sigma\sqrt{t} + \|H\|_\infty)e^{4\Lambda_b t}  + \frac{1}{2}g_0( t) = g_0(t), \quad t\in[0, T_0],
\end{align*}
where $T_0 = (2\|G\|_\infty)^{-1}$ and $g_0$ is given by \eqref{definition g_0}.  In other words \eqref{sup bound} holds on $[0, T_0]$ for the stated $T_0$ and $g_0$.

\vspace{0.2cm}
\noindent\textit{Step 2: Long time.}
Suppose now that \eqref{sup bound} holds on some interval $[0, \bar{T}]$ with $g=\bar{g}$ for some function $\bar{g}:[0, \bar{T}] \to [0, \infty)$.  We aim to show that it holds on $[0, \bar{T} + T_0]$, where $T_0$ is as in Step 1. 

To this end, suppose $h\in \mathcal{C}^1([0, \bar{T} + T_0])$ is non-decreasing and such that $h(0)=0$ and $h(t)\leq g(t)$ for $t\in[0, \bar{T} + T_0]$, where $g(t) = \bar{g}(t)$ for $t\leq \bar{T}$ and $g$ will be defined for $t\in(\bar{T}, \bar{T} + T_0]$ below.
Define $\hat{Z}^{h}_s := Z^{h}_{s+\bar{T}}$ for all $s\in [0, T_0]$.  Then, by the Markov property, the dynamics of $\hat{Z}^{h}$ are given by
\[
\hat{Z}^{h}_s  = Z^{h}_{\bar{T}}+ f^{\sharp{\bar{T}}}_{h}(s) + \int_{0}^{s} b(\hat{Z}^{h}_r -\hat{M}^{h}_r)dr + \int_0^s \sigma(\hat{Z}^{h}_r -\hat{M}^{h}_r) d\hat{W}_r ,
\]
where $(\hat{W}_s)_{s\geq0}$ is a Brownian motion, $\hat{M}^{h}_s = \sum_{k\geq 1}\Ind_{[0, s]}(\hat{\tau}^{h}_k)$,
and $\hat{\tau}^{h}_k = \inf\{s\geq0: \hat{Z}^{h}_{t}\geq k\}$, $k\geq 1$ with $\hat{\tau}^{h}_0 = 0$.
Here, as in \eqref{fshift}, 
\[
f_h^{\sharp \bar{T}}(s) := f_h(s+ \bar{T} ) - f_h(\bar{T}), \qquad s\in [0, T_0].
\]

Now, proceeding as in Step 1, 
if we define $\hat{F}(s) := \E( \sup_{r \in[0,s]}|\hat{Z}_{r}^{h}| )$, by the Burkh\"older-Davis-Gundy inequality,
\begin{equation*}
\hat{F}(s) \leq \E|Z_{\bar{T}}^h| +\|f_h^{\sharp \bar{T}}\|_{\infty, s}  + \Lambda_bs  + c\Lambda_\sigma\sqrt{s} + 2\Lambda_b \int_{0}^{s}\hat{F}(u)  du.
\end{equation*}
Moreover, again by \eqref{fh}, using the facts that $h$ is non-decreasing, $h(0) =0$ and $G(0) =0$, we have for $s\in [0, T_0]$
\begin{align*}
\|f_h^{\sharp \bar{T}}\|_{\infty, s} 
&\leq 2\|H\|_\infty + \sup_{r\in[0,s]}\left|\int_{\bar{T}}^{r + \bar{T}}\frac{d}{du}\left[\int_0^uG(u - v)h(v)dv\right]du\right| \\
&\leq 2\|H\|_\infty + \int_{\bar{T}}^{s + \bar{T}}\int_0^u|G(u - v)|h'(v)dvdu \\
&\leq 2\|H\|_\infty + s\|G\|_\infty h(r+\bar{T}) \leq 2\|H\|_\infty + \frac{1}{2} h(s+\bar{T}),
\end{align*}
for all $s\in[0, T_0]$, recalling that $T_0 = (2\|G\|_{\infty})^{-1}$. Since we have assumed that $h(t)\leq g(t)$ for $t\in[0, \bar{T} + T_0]$, we then have
\begin{equation*}
\hat{F}(s) \leq \E|Z_{\bar{T}}^h| + \Lambda_bs  + c\Lambda_\sigma\sqrt{s} + 2\Lambda_b \int_{0}^{s}\hat{F}(u)  du + 2\|H\|_\infty + \frac{1}{2}g(s+\bar{T}) ,
\end{equation*}
for $s\in [0, T_0]$.  Now, by assumption, \eqref{sup bound} is true on $[0, \bar{T}]$ with $g=\bar{g}$, and $h(t)\leq \bar{g}(t)$ on $[0, \bar{T}]$.  Thus  $\E|Z_{\bar{T}}^h| \leq \bar{g}(\bar{T})$, so that 
\begin{equation*}
\hat{F}(s) \leq  \bar{g}(\bar{T})+ \Lambda_bs  + c\Lambda_\sigma\sqrt{s} + 2\Lambda_b \int_{0}^{s}\hat{F}(u)  du + 2\|H\|_\infty + \frac{1}{2}g(s+\bar{T}),
\end{equation*}
for $s\in [0, T_0]$.
Arguing in exactly the same way as at the end of Step 1, we can then deduce that if 
\begin{equation}
\label{definition g lt}
g(t) := 2\big(\bar{g}(\bar{T}) + \Lambda_b(t - \bar{T}) + c\Lambda_\sigma\sqrt{t - \bar{T}} + 2\|H\|_\infty \big)e^{4\Lambda_b(t-\bar{T})},
\end{equation}
for $t\in (\bar{T}, \bar{T}+T_0]$, then 
\begin{equation*}
\hat{F}(s) \leq g(s+\bar{T}),
\end{equation*}
for $s\in [0, T_0]$.  Thus, we have found $g$ (given by $\bar{g}$ on $[0, \bar{T}]$ and \eqref{definition g lt} on  $(\bar{T}, \bar{T}+T_0]$) such that \eqref{sup bound} holds on $[0, \bar{T} + T_0]$.
The claim \eqref{sup bound} then clearly holds on any interval, since we can iterate this second step.
\qed


\section{Proof of Theorem \ref{thm: convergence}: Convergence of the particle system}
\label{sec:convergence}

The purpose of this section is to complete our study, and prove the second main result, namely Theorem \ref{thm: convergence}.
In a similar way to Section \ref{sec: e+u}, it is convenient to reformulate the particle system \eqref{particle system} in terms of the processes $(Z^i_t)_{t\geq0} := (U^i_t + M^i_t)_{t\geq 0}$ for $i\in\{1, \dots, N\}$.  This is because for each $i\in\{1, \dots, N\}$, the dynamics of $Z^i$ are given by
\begin{equation}
\label{Z particle}
\begin{cases}
Z_t^i=U^i_0 + H(t)  + \int_0^t b(Z_s^i - M_s^i)ds +\sum_{j=1}^N\frac{J_{ij}}{S^N_i}\int_0^t   {G}(t -s)M_s^j ds + \int_0^t\sigma(Z_s^i - M_s^i)dW_s^i,\\
M_t^i =\lfloor( \sup_{s\in [0, t]}Z^{i}_s)_+ \rfloor,\\
\tau_k^i= \inf\{t\geq0:  Z_{t}^i \geq k\},\ k\in\mathbb{N}\backslash\{0\},\ \tau^i_0=0,
\end{cases}
\end{equation}
where as above $\lfloor x \rfloor$ and $(x)_+$ indicate the integer part of $x$ and $\max\{x, 0\}$ respectively, for any $x\in\R$. This reformulation is seen in exactly the same way as the reformulation of the limit equation in \eqref{Z limit equation}. As there, again the point is that each $Z^i$ has a continuous path, but the system is equivalent to the original one \eqref{particle system}.  This will eliminate the need to work with Skorohod topologies (at least until the very end of the proof).

\subsection{Notation}
Fix $T>0$ and let $(Z_t^i)_{t\in[0, T], i\in\{1, \dots, N\}}$ be the solution to the system \eqref{Z particle}.  Define the weighted empirical measure
\begin{equation}
\label{mu bar}
\bar{\mu}^N_i := \frac{1}{S^N_i}\sum_{j=1}^NJ_{ij}\delta_{Z^{j}},
\end{equation}
which is a random probability measure on $\mathcal{C}([0, T])$ i.e. $\bar{\mu}^N_i \in \mathcal{P}\left(\mathcal{C}([0, T])\right)$ for every $i\in\{1, \dots, N\}$.  We also define
\begin{equation}
\label{Pi def}
\Pi^N_i := \mathrm{Law}(\bar{\mu}^N_i), \quad i\in\{1, \dots, N\},
\end{equation}
which is now a probability measure on the space of probability measures $\mathcal{P}\left(\mathcal{C}([0, T])\right)$.  In view of the statement of Theorem  \ref{thm: convergence}, the aim is to study the convergence of $(\Pi^N_i)_{N\geq i}$ for each $i\in \mathbb{N}$.

Throughout this section we will denote by $z=(z_t)_{t\in[0, T]}$ the canonical process on $\mathcal{C}([0, T])$ equipped with the usual uniform topology.  For $t\in[0, T]$ we moreover define
 \begin{equation}
 \label{mn}
 m_t = m_t(z) := \Big\lfloor\Big(\sup_{s\in[0,t]} z_s\Big)_+\Big\rfloor, \quad n_t  = n_t(z) := \int_0^t{G}(t-s)m_s(z) ds,
 \end{equation}
 for any $z\in \mathcal{C}([0, T])$. 
With this notation in place we can re-write \eqref{Z particle} as
\begin{equation}
\label{particle system em}
\begin{cases}
Z_t^i=U^i_0 + H(t)  + \int_0^t b(Z_s^i - M_s^i)ds +\left \langle \bar{\mu}^N_i,n_t\right \rangle + \int_0^t\sigma(Z_s^i - M_s^i)dW_s^i,\\
M_t^i = m_t(Z^i),\\
\tau_k^i= \inf\{t\geq0:  Z_{t}^i \geq k\},\ k\in\mathbb{N}\backslash\{0\},\ \tau_0^i=0,
\end{cases}
\end{equation}
for $t\in[0, T]$ and $i\in\{1, \dots, N\}$.  Here $\langle \bar{\mu}^N_i, n_t\rangle$ stands for the expectation of $n_t$ given by \eqref{mn}
under the measure $\bar{\mu}^N_i$ i.e. $ \langle \bar{\mu}^N_i,n_t \rangle = \int_{\mathcal{C}([0, T])} n_t(z)\bar{\mu}^N_i(dz)$.

\subsection{Tightness of the law of the weighted empirical measure}
In this section we show that for each $i\in\mathbb{N}$, the family of measures $(\Pi^N_i)_{N\geq i}$ is tight.  The first result we need is a bound on the moments of the solution to the system \eqref{Z particle}.
\begin{lem}
\label{lem: moment bounds}
Let $p\geq 1$.  Then, under the conditions of Theorem \ref{thm: convergence}, for any $T> 0$ there exists a constant $C_T^{( p)}$ (independent of $N$)
such that for any $i\in\{1, \dots, N\}$
\[
\E\left(\left[M^{i}_T\right]^p\right)  \leq \E\Big(\Big[\sup_{t\in[0, T]}|Z^{i}_t|\Big]^p\Big) \leq C_T^{( p)}.
\]
\end{lem}

\begin{proof}
Let $i\in\{1, \dots, N\}$, and $T>0$. By \eqref{Z particle} and the Lipschitz property of $b$ we have
\begin{align}
\label{PS lem: AP bound 2}
\sup_{t\in[0, T]}|Z^{i}_t| &\leq |U^{i}_{0}| +\|H\|_\infty + \Lambda_bT + 2\Lambda_b \int_{0}^T\sup_{r\in [0, s]}\left| Z^{i}_r\right|ds\nonumber\\
&\qquad +\|G\|_{\infty}\int_0^T\frac{1}{S^N_i}\sum_{j=1}^NJ_{ij}M_s^{j}ds  + \sup_{t\in[0, T]}\left|\int_0^t\sigma(Z^{i}_s - M^{i}_s )dW_s^i\right|.
\end{align}

Now let $p\geq1$.  Raising both sides of \eqref{PS lem: AP bound 2} to the power $p$ and using Jensen's inequality yields the existence of a constant $C_T^{( p)}$ (depending only on $T$, $p$, and the bounds in Assumptions \ref{assumptions}) such that
\begin{align}
\label{PS lem: AP bound 2 p}
\Big[\sup_{t\in[0, T]}|Z^{i}_t|\Big]^p & \leq C_T^{( p)}\Big(1 + \int_{0}^T\Big[\sup_{r\in [0, s]}\left| Z^{i}_r\right|\Big]^pds+\int_0^T\frac{1}{S^N_i}\sum_{j=1}^NJ_{ij}\left[M_s^{j}\right]^pds \nonumber\\
&\qquad\qquad\qquad\qquad + \Big[\sup_{t\in[0, T]}\Big|\int_0^t\sigma(Z^{i}_s - M^{i}_s )dW_s^i\Big|\Big]^p\Big),
 \end{align}
almost surely. 
Thus, by the Burkh\"older-Davies-Gundy inequality,
\begin{align*}
\max_{i\in\{1, \dots, N\}}\E\Big(\Big[\sup_{t\in[0, T]}|Z^{i}_t|\Big]^p\Big) & \leq C_T^{( p)}\Big(1 + \int_{0}^T\max_{i\in\{1, \dots, N\}}\E\Big(\Big[\sup_{r\in [0, s]}\left| Z^{i}_r\right|\Big]^p\Big)ds\Big),
 \end{align*}
where we have increased the constant $ C_T^{( p)}$ (so that it now depends on $\Lambda_\sigma$) and used the fact that $M_s^{i} \leq \sup_{r\in[0, s]}|Z_r^i|$ for any $s\geq0$ and $i\in\{1,\dots, N\}$.  Gronwall's lemma then yields the result.
\end{proof}

\begin{prop}
\label{tightness}
Under the conditions of Theorem \ref{thm: convergence}, for each $i\in\mathbb{N}$ the family $(\Pi^N_i)_{N\geq i}$ is tight in $ \mathcal{P}(\mathcal{P}(\mathcal{C}([0, T])))$.
\end{prop}

\begin{proof}
Due to the non-exchangeability of the particle system, we need to be a little careful.  Let $i\in\mathbb{N}$.  Then according to \cite[Proposition 4.6]{Meleard}, the family $(\Pi^N_i)_{N\geq i}$ is tight if and only if the sequence of intensity measures $(I(\Pi^N_i))_{N\geq i}\subset \mathcal{P}(\mathcal{C}([0, T]))$ is tight, where for $N\geq i$, $I(\Pi^N_i)$ is defined by
\[
I(\Pi^N_i)(A) := \int \langle \mu, \Ind_A\rangle \Pi^N_i(d\mu) = \frac{1}{S_i^N}\sum_{j=1}^NJ_{ij}\mathbb{P}\left(Z^{j}\in A\right),
\]
for any $A\in\mathcal{B}(\mathcal{C}([0, T]))$.  It thus suffices to show that $(I(\Pi^N_i))_{N\geq i}$ is tight in $\mathcal{P}\left(\mathcal{C}([0, T])\right)$.  By \cite[Theorem 4.10]{KS}, this in turn is equivalent to showing that
\[
\lim_{K\to\infty} \sup_{N\geq i}I(\Pi^N_i)\left\{z \in \mathcal{C}([0, T]) : |z(0)|>K\right\} = 0,
\]
and 
\[
\lim_{\delta\to 0} \sup_{N\geq i}I(\Pi^N_i)\left\{z \in \mathcal{C}([0, T]) : \sup_{\substack{|s-t|\leq \delta \\0\leq s,t \leq T}}|z(t)-z(s)|>\varepsilon \right\} = 0, \quad \forall\varepsilon>0.
\]
The first condition is clearly satisfied under our assumptions on $U^i_0$.  By the definition of $I(\Pi^N_i)$, the second condition reads
\begin{equation}
\label{prop: tightness 0}
\lim_{\delta\to 0} \sup_{N\geq i}\frac{1}{S^N_i}\sum_{j=1}^NJ_{ij} \mathbb{P}\left\{\sup_{\substack{|s-t|\leq \delta \\0\leq s,t \leq T}}|Z^{j}_t-Z^{j}_s|>\varepsilon \right\} = 0, \qquad \forall \varepsilon>0.
\end{equation}
To check this holds, let $\varepsilon>0$.  Then we have 
\begin{align}
\label{prop: tightness 1}
&\frac{1}{S^N_i}\sum_{j=1}^NJ_{ij} \mathbb{P}\left\{\sup_{\substack{|s-t|\leq \delta \\0\leq s,t \leq T}}|Z^{j}_t-Z^{j}_s|>\varepsilon \right\} \leq \frac{1}{S^N_i}\sum_{j=1}^NJ_{ij} \frac{1}{\varepsilon} \E\left(\sup_{\substack{|s-t|\leq \delta \\0\leq s,t \leq T}}|Z^{j}_t-Z^{j}_s|\right).
\end{align}
To complete the proof, we can deduce by standard arguments (and using the Assumptions \ref{assumptions}) that
\begin{align}
&\E\Big(\sup_{|s-t|\leq \delta,\ 0\leq s,t \leq T}|Z^{i}_t -  Z^{i}_s|\Big) \leq C_T\sqrt{\delta}
 \end{align}
 for all $N\geq i$ and $\delta\in(0,1)$, where $C_T$ is some constant (independent of $i$, $N$ and $\delta$).  Using this in \eqref{prop: tightness 1}, \eqref{prop: tightness 0} certainly holds, and the lemma is proved.
\end{proof}

\subsection{Proof of Theorem \ref{thm: convergence}}
\label{sec: proof of convergence}
By Proposition \ref{tightness} we now know that, for each $i\in \mathbb{N}$, the family $(\Pi^N_i)_{N\geq i}$ has a subsequence (which we also denote by $(\Pi^N_i)_{N\geq i}$) that is weakly convergent to some 
\[
\Pi^\infty_i\in\mathcal{P}\left(\mathcal{P}\left(\mathcal{C}([0, T])\right)\right),
\]
as $N\to\infty$.

\vspace{0.3cm}
\noindent\textit{First step:}
We aim to show that, under the stated conditions of Theorem \ref{thm: convergence}, for every  $i\in \mathbb{N}$
\begin{equation}
\label{Pi infty}
\Pi^\infty_i = \delta_{\mathrm{Law}((Z_t)_{t\in[0, T]})},
\end{equation}
where $(Z_t)_{t\in[0, T]}$ is the unique solution to the limit equation \eqref{Z limit equation}, which is guaranteed to exist by Theorem \ref{thm: existence and uniqueness}.

We in fact show that for almost all $\mu$ under $\Pi^\infty_i$, $\mu$ solves the nonlinear martingale problem associated to the unique solution of \eqref{Z limit equation}.
To this end, we introduce the following operator. For $t\in[0, T]$, $\mu \in \mathcal{P}\left(\mathcal{C}([0, T])\right)$, and $\varphi\in\mathcal{C}_b^2(\R)$ define, for $z \in\mathcal{C}([0, T])$,
\begin{equation}
\label{generator}
\mathcal{L}_{t, \mu} \varphi(z) := \frac{1}{2}\sigma^2(z_t-m_t(z))\varphi''(z_t) + \varphi'(z_t)\left(H'(t) + b(z_t - m_t(z)) + \frac{d}{dt}\langle \mu, n_t(z)\rangle \right),
\end{equation}
where $m_t(z)$ and $n_t(z)$ are defined by \eqref{mn}.  
We will say that $\mu\in\mathcal{P}\left(\mathcal{C}([0, T])\right)$ solves the nonlinear martingale problem associated with \eqref{Z limit equation} if
\begin{equation}
\label{nonlinear mart}
\left(\varphi(z_t) - \varphi(z_0) - \int_0^t\mathcal{L}_{r, \mu} \varphi(z_r)dr\right)_{t\geq0}
\end{equation}
is a martingale under $\mu$, whenever $\varphi\in\mathcal{C}_b^2(\R)$.  By  pathwise uniqueness for the solution $(Z_t)_{t\in[0, T]}$ to \eqref{Z limit equation} (Theorem \ref{thm: existence and uniqueness}), the unique solution to the nonlinear martingale problem must be the law of $(Z_t)_{t\in[0, T]}$ (see for example \cite[Theorem 1.1]{Bass}).  Thus if we can show that for almost all $\mu$ under $\Pi^\infty_i$, $\mu$ solves the nonlinear martingale problem, \eqref{Pi infty} must hold.

To proceed, for a fixed $\varphi\in\mathcal{C}_b^2(\R)$ and $0\leq s\leq t\leq T$, define the functional
\[
F(\mu): = \left\langle \mu, \left(\varphi(z_t) - \varphi(z_s) - \int_s^t\mathcal{L}_{r, \mu} \varphi(z_r)dr\right)\psi(z_{s_1},\dots, z_{s_k}) \right\rangle,
\]
for $\mu\in \mathcal{P}\left(\mathcal{C}([0, T])\right)$. Here $k\in\mathbb{N}$, $\psi\in\mathcal{C}_b(\R^k)$ and $0 < s_1 < \dots < s_k \leq s \leq t \leq T$.  
We then have by It\^o's formula that for $i\in\{1, \dots, N\}$
\begin{align*}
&\E\left(\left[F(\bar{\mu}^N_i)\right]^2\right) = \E\left(\left\langle \bar{\mu}^N_i, \left(\varphi(z_t) - \varphi(z_s) - \int_s^t\mathcal{L}_{r, \bar{\mu}^N_i} \varphi(z_r)dr\right)\psi(z_{s_1},\dots, z_{s_k}) \right\rangle^2\right)\\
&= \E\left(\left[ \frac{1}{S_i^N}\sum_{j=1}^NJ_{ij}\left(\int_s^t\varphi'(Z^j_r)\sigma(Z^j_r - M^j_r)dW_r^j\right)\psi(Z^j_{s_1},\dots, Z^j_{s_k})\right]^2 \right)\\
&\leq \|\psi\|_\infty^2 \E\left(\left[ \frac{1}{S_i^N}\sum_{j=1}^NJ_{ij}\left(\int_s^t\varphi'(Z^j_r)\sigma(Z^j_r - M^j_r)dW_r^j\right)\right]^2 \right).
\end{align*}
Now, by the independence of $(W^j)_{j\in\{1, \dots, N\}}$ (and the progressive measurability of $Z^j - M^j$ for all $j\in\{1, \dots, N\}$), we have that 
\begin{align*}
& \E\left(\left[ \frac{1}{S_i^N}\sum_{j=1}^NJ_{ij}\left(\int_s^t\varphi'(Z^j_r)\sigma(Z^j_r - M^j_r)dW_r^j\right)\right]^2 \right) \\
&\qquad\qquad =\frac{1}{(S_i^N)^2}\sum_{j=1}^NJ^2_{ij}\E\left(\left(\int_s^t\varphi'(Z^j_r)\sigma(Z^j_r - M^j_r)dW_r^j\right)^2\right).
\end{align*}
Thus
\begin{align*}
\E\left(\left[F(\bar{\mu}^N_i)\right]^2\right) 
&\leq \|\psi\|_\infty^2\|\varphi'\|^2_\infty\Lambda_\sigma^2(t-s) \frac{1}{(S_i^N)^2}\sum_{j=1}^NJ^2_{ij} \to 0,
\end{align*}
where we use assumption \eqref{J condition}.  We have thus shown that
\begin{equation}
\label{Nlimit}
\lim_{N\to\infty}\int \left[F(\mu)\right]^2\Pi^N_i(d\mu) = 0, \quad \forall i\in\mathbb{N}.
\end{equation}
In order to proceed, we need the following crucial lemma.

\begin{lem}[Crucial lemma]
\label{lem:functional continuity}
For every $i\in\mathbb{N}$, $\varphi\in\mathcal{C}_b^2(\R)$, $\psi\in\mathcal{C}_b(\R^k)$ and $0 < s_1 < \dots < s_k \leq s \leq t \leq T$, the functional on $\mathcal{P}\left(\mathcal{C}([0, T])\right)$ given by
\[
\mu \mapsto\left\langle \mu, \left(\varphi(z_t) - \varphi(z_s) - \int_s^t\mathcal{L}_{r, \mu} \varphi(z_r)dr\right)\psi(z_{s_1},\dots, z_{s_k}) \right\rangle
\]
is $\Pi^\infty_i$-almost everywhere continuous.
\end{lem}

We delay the proof of this result until the next section, and devote the rest of this section to completing the proof of Theorem \ref{thm: convergence}.
 Given Lemma \ref{lem:functional continuity}, it then follows from \eqref{Nlimit} and the continuous mapping theorem (see for example \cite[Theorem 2.7]{Billingsley}) that for each $i\in\mathbb{N}$
\[
 \int  \left[F(\mu)\right]^2 \Pi^\infty_i(d\mu) = 0.
\]
We can then conclude that 
\[
F(\mu) = \left\langle \mu, \left(\varphi(z_t) - \varphi(z_s) - \int_s^t\mathcal{L}_{r, \mu} \varphi(z_r)dr\right)\psi(z_{s_1},\dots, z_{s_k}) \right\rangle = 0,
\]
for almost all $\mu$ under $\Pi^\infty_i$.  In other words, for almost all $\mu$ under $\Pi^\infty_i$, $\mu$ solves the nonlinear martingale problem \eqref{nonlinear mart}, which as observed above, has $\mathrm{Law}((Z_t)_{t\in[0, T]})$ as its unique solution.  We have thus shown that, given Lemma  \ref{lem:functional continuity}, \eqref{Pi infty} does indeed hold.


\vspace{0.3cm}
\noindent\textit{Second step:}
To complete the proof of Theorem \ref{thm: convergence}, it remains to deduce from the above that
\begin{equation}
\label{end claim}
\mathrm{Law}\left(\frac{1}{S^N_i}\sum_{j=1}^NJ_{ij}\delta_{U_j}\right) \Rightarrow \delta_{\mathrm{Law}\left(\left(U_t\right)_{t\in[0, T]}\right)},
\end{equation}
as $N\to\infty$ for all $i\in\mathbb{N}$, where the convergence is now in the weak sense in the space $\mathcal{P}(\mathcal{P}(\mathcal{D}([0, T])))$, and $(U_t)_{t\in[0, T]}$ is the unique solution to the original limit equation \eqref{limit equation}.  To this end, the weak convergence of $\Pi_i^N$ to $\Pi_i^\infty = \delta_{\mathrm{Law}((Z_t)_{t\in[0, T]})}$ says that (by the portmanteau theorem)
\[
\int\langle \mu, \Psi \rangle \Pi_i^N(d\mu) \to \int\langle \mu, \Psi \rangle \Pi_i^\infty(d\mu) = \E(\Psi(Z)),
\]
for all bounded $\Psi:\mathcal{C}([0, T])\to\R$ that are almost surely continuous on \linebreak$\mathrm{supp}(\mathrm{Law}((Z_t)_{t\in[0, T]}))$.
Now suppose that $\tilde{\Psi}:\mathcal{D}([0, T])\to\R$ is bounded and almost surely continuous on $\mathrm{supp}(\mathrm{Law}((U_t)_{t\in[0, T]}))$, where
$U_t := Z_t - M_t$.  Admit for the moment that the bounded map $\Psi:\mathcal{C}([0, T])\to\R$ defined by
\[
\Psi(z) := \tilde{\Psi}(z - m(z)),\qquad z\in  \mathcal{C}([0, T]),
\]
where $m(z)$ is given by \eqref{mn}, is almost surely continuous under \linebreak $\mathrm{supp}(\mathrm{Law}((Z_t)_{t\in[0, T]}))$.
Then
\[
\E\left(\frac{1}{S^N_i}\sum_{j=1}^NJ_{ij}\tilde{\Psi}(Z^j - M^j) \right) = \E\left(\frac{1}{S^N_i}\sum_{j=1}^NJ_{ij}\tilde{\Psi}(U^j) \right)   \to \E\left(\tilde{\Psi}(U)\right)
\]
as $N\to\infty$ for all bounded $\tilde{\Psi}:\mathcal{D}([0, T])\to\R$ that are almost surely continuous on $\mathrm{supp}(\mathrm{Law}((U_t)_{t\in[0, T]}))$.  

To complete the proof of the claim \eqref{end claim}, and hence Theorem \ref{thm: convergence}, it remains to justify that $\Psi$ is almost surely  continuous under  $\mathrm{supp}(\mathrm{Law}((Z_t)_{t\in[0, T]}))$, for which it suffices to show that $z^k \to z$ in $\mathcal{C}([0, T])$ implies $z^k - m(z^k) \to z - m(z)$ in $\mathcal{D}([0, T])$ for all $z \in \mathrm{supp}(\mathrm{Law}((Z_t)_{t\in[0, T]}))$.  To see this, set $\bar{z}^k= z^k - m(z^k)$, $\bar{z} = z - m(z)$.   To show that $\bar{z}^k\to\bar{z}$ in $\mathcal{D}([0, T])$ we must check (a), (b) and (c) of  \cite[Chapter 3, Proposition 6.5]{Ethier-Kurtz}.  Let $(t_k)\subset[0, T]$ be such that $\lim_{k\to\infty}t_k = t\in[0, T]$.  The only possible problem points are when $t\in I$, the countable set of hitting times of $1$ by $\bar{z}$ (see Lemma \ref{lem:crossing property}), so without loss of generality suppose $t$ is the first of these hitting times.  To see that (a) holds at $t$, either ${z}^k_{t_k} <1$ for all $k$, in which case 
$z^k_{t_k} = \bar{z}^k_{t_k} \to \bar{z}_{t-} = {z}_{t} =1$, or $\exists k$ such that ${z}^k_{t_k} \geq 1$, in which case for $k$ large enough
$|\bar{z}^k_{t_k} - \bar{z}_t| = |{z}^k_{t_k} - 1| = |{z}^k_{t_k} - z_{t}| \to0$ since $z^k\to z$ in $\mathcal{C}([0, T])$.  Thus 
$|\bar{z}^k_{t_k} - \bar{z}_t| \wedge |\bar{z}^k_{t_k} - \bar{z}_{t-}|\to0$ as $k\to \infty$.  Points (b) and (c) follow similarly.

\subsection{Proof of crucial Lemma \ref{lem:functional continuity}}
The proof of Lemma \ref{lem:functional continuity} follows \cite[Lemma 5.10]{DIRT2}, but we must adjust for the fact that here we are working with a general diffusion coefficient, rather than a constant.  In counterpart, the proof in our case simplifies in places due to the fact \rev{that} we work on the space of continuous functions with the usual topology (rather than on the Skorohod space).

Fix $i\in \mathbb{N}$. Let $(\mu^l)_{l\geq 1}\subset \mathrm{supp}(\Pi^\infty_i)$ be a sequence converging to   $\mu\in \mathrm{supp}(\Pi^\infty_i)$ in the weak sense.  The lemma will be proved if we can show that
\begin{align*}
&\lim_{l\to\infty} \left\langle\mu^{l}, \left(\varphi(z_t) - \varphi(z_s) - \int_s^t\mathcal{L}_{r, \mu^{l}} \varphi(z_r)dr\right)\psi(z_{s_1},\dots, z_{s_k}) \right\rangle \\
&\qquad = \left\langle\mu, \left(\varphi(z_t) - \varphi(z_s) - \int_s^t\mathcal{L}_{r, \mu} \varphi(z_r)dr\right)\psi(z_{s_1},\dots, z_{s_k})\right\rangle.
\end{align*}
By the definition of $\mathcal{L}_{r, \mu}$, since $\varphi\in\mathcal{C}^2_b(\R)$, $\psi\in\mathcal{C}_b(\R^k)$ and by Assumptions \ref{assumptions}, one can see that it suffices to show that
\begin{equation}
\label{show 1}
\lim_{l\to\infty} \left\langle\mu^{l}, \int_s^tm_r(z)dr\right\rangle = \left\langle\mu, \int_s^tm_r(z)dr\right\rangle,
\end{equation}
where $m_r$ is defined by \eqref{mn}.
We need the following lemma:

\begin{lem}
\label{lem:crossing property}
Let $i\in \mathbb{N}$.  For every $\mu \in \mathrm{supp}(\Pi^\infty_i) \subset \mathcal{P}(\mathcal{C}([0, T]))$, the following crossing property is satisfied: for every integer $k\geq1$ and every $\varepsilon>0$
\begin{equation}
\label{crossing property}
\mu\left\{z\in \mathcal{C}([0, T]): \tau_k(z)<T,\ \sup_{t\in [\tau_k(z), (\tau_k(z) + \varepsilon)\wedge T)}(z_t - k) = 0 \right\} = 0,
\end{equation}
where $\tau_k(z) = \inf\{t\in[0, T]: z_t \geq k\}$ ($\inf\emptyset = T$).
\end{lem}

We prove this lemma in the next section, but first show that it is enough to complete the proof of Lemma \ref{lem:functional continuity}.
Indeed, let $z\in\mathrm{supp}(\mu)$, and let $(z^k)_{k\geq1}\subset \mathcal{C}([0, T])$ be a sequence converging to $z$ in  $\mathcal{C}([0, T])$.  Thanks to Lemma \ref{lem:crossing property}, we have that there exists a countable set $I\subset [0, T]$ (consisting of all times at which $z$ crosses a new integer and possibly $T$) such that 
\[
m_r(z^k) \to m_r(z), \qquad r\in[0, T]\backslash I,
\]
almost surely. \rev{The point is that Lemma \ref{lem:crossing property} rules out the possibility of $z$ touching but not crossing any integer.  If this were possible, there could be a subset of $[0, T]$ of positive Lebesgue measure upon which $m_r(z^k) \not\to m_r(z)$, and Lemma \ref{lem:functional continuity} would not hold.}

Thus, the application $z\mapsto \int_s^t m_r(z)dr$ is continuous almost surely for all $z$ in the support of $\mu$.  Therefore, by the continuous mapping theorem \cite[Theorem 2.7]{Billingsley}, \eqref{show 1} holds.
\qed

\subsection{Proof of Lemma \ref{lem:crossing property}}
Let $i\in \mathbb{N}$.  The proof is again adapted from \cite{DIRT2}, though we again must generalize it for a general diffusion coefficient $\sigma$.
To proceed, we introduce the coupled weighted empirical measure
\[
\bar{\theta}^N_i := \frac{1}{S^N_i}\sum_{j=1}^NJ_{ij}\delta_{(Z^{j}, \int_0^\cdot \sigma(Z^j_s - M^j_s)dW^j_s)} \in \mathcal{P}\left(\mathcal{C}([0, T])\times\mathcal{C}([0, T])\right),
\]
where $(Z^j)_{j\in\{1, \dots, N\}}$ is the solution to the system \eqref{Z particle} as usual and $(W^j)_{j\in\{1, \dots, N\}}$ are the corresponding independent Brownian motions driving the system. Note that the marginal of $\bar{\theta}^N_i$ on the first coordinate space of $\mathcal{C}([0, T])\times\mathcal{C}([0, T])$ is $\bar{\mu}^N_i$ given by \eqref{mu bar}.

We also define
\[
\Xi^N_i := \mathrm{Law}(\bar{\theta}^N_i), \quad i\in\{1, \dots, N\},
\]
so that $\Xi^N_i\in \mathcal{P}\left[\mathcal{P}\left(\mathcal{C}([0, T])\times\mathcal{C}([0, T])\right)\right]$.  In a very similar way to Lemma \ref{tightness}, one can see that $(\Xi^N_i)_{N\geq i}$ is tight, for any $i\in\mathbb{N}$.  We can thus extract a weakly convergent subsequence, still denoted by $(\Xi^N_i)_{N\geq i}$ which converges to some $\Xi^\infty_i$.  
\rev{Our objective now is to show that the first marginal of $\Xi^\infty_i$ satisfies property \eqref{crossing property}.}

\vspace{0.3cm}
\noindent\textit{Step 1:}  
By definition of the particle system \eqref{Z particle}, for any $0\leq s \leq t \leq T$ and $N\geq i$ we have
\begin{align*}
Z^{i}_t -  Z^{i}_s 
&\geq - c_T(t-s)\left(1 +  \max_{j\in\{1, \dots, N\}}\sup_{r\leq T}|Z^{j}_r|\right) + \int_s^t\sigma(Z^{i}_r - M^{i}_r)dW^i_r
 \end{align*}
almost surely, for a constant $c_T>0$ that depends only on $T$ and the bounds in Assumptions \ref{assumptions}.
Thus 
\begin{align*}
 \frac{1}{S^N_i}\sum_{j=1}^NJ_{ij}\left(Z^{j}_t -  Z^{j}_s\right)
&\geq - c_T(t-s)\left(1 + \mathcal{Z}_T\right) + \frac{1}{S^N_i}\sum_{j=1}^NJ_{ij}\int_s^t\sigma(Z^{j}_r - M^{j}_r)dW^j_r,
 \end{align*}
almost surely, where for convenience we have used the notation $\mathcal{Z}_T:= \max_{j}\sup_{r\leq T}|Z^{j}_r|$.  Hence
\begin{align*}
1 &= \mathbb{P}\left(\sum_{j=1}^N\frac{J_{ij}}{S^N_i}\left(Z^{j}_t -  Z^{j}_s\right)\geq \sum_{j=1}^N\frac{J_{ij}}{S^N_i}\int_s^t\sigma(Z^{j}_r - M^{j}_r)dW^j_r - c_T(t-s)\left(1 +  \mathcal{Z}_T\right) \right) \\
 & \leq \mathbb{P}\left(\sum_{j=1}^N\frac{J_{ij}}{S^N_i}\left(Z^{j}_t -  Z^{j}_s\right)\geq   \sum_{j=1}^N\frac{J_{ij}}{S^N_i}\int_s^t\sigma(Z^{j}_r - M^{j}_r)dW^j_r - c_T(t-s)(1+K)\right) \\
 &\qquad\qquad +
 \mathbb{P}\left(\mathcal{Z}_T\geq K\right),
  \end{align*}
  for any $K\geq0$.  By Lemma \ref{lem: moment bounds}, we have that
 \[
  \mathbb{P}\left(\mathcal{Z}_T\geq K\right) \leq \frac{1}{K}\max_{j\in\{1, \dots, N\}}\E\left(\sup_{r\leq T}|Z^{j}_r|\right) \leq \frac{C_T}{K}
 \]
 for a constant $C_T$ independent of $N$.  Thus for any $i\in \mathbb{N}$ and $N\geq i$ we have shown that
 for almost all $\theta$ under $\Xi_i^N$ it holds that
\begin{equation}
\label{theta bound}
 \theta\left\{ z_t - z_s \geq w_t - w_s - c_T(t-s)(1 +K),\ \forall 0\leq s\leq t\leq T\right\} \geq 1 - \frac{C_T}{K}
 \end{equation}
 for all $K\geq0$, where now $(z,w)$ is the canonical process on $\mathcal{C}([0, T])\times\mathcal{C}([0, T])$. We then claim that the same must also be true for all $\theta$ under $\Xi_i^\infty$.  Indeed, \eqref{theta bound} reads
 \[
 \Xi_i^N \Big\{\theta: \theta\left\{ z_t - z_s \geq w_t - w_s - c_T(t-s)(1 +K),\ \forall 0\leq s\leq t\leq T\right\} \geq 1 - \frac{C_T}{K}\Big\} = 1.
 \]
 Since $\Xi_i^N \to\Xi_i^\infty$ in the weak sense, by the portmanteau theorem, $\limsup_N \Xi_i^N(\mathcal{A}) \leq \Xi_i^\infty(\mathcal{A})$ for all closed $\mathcal{A}$.  Choose $ \mathcal{A} := \{\theta: \theta(\mathcal{B}_K) \geq 1 - \frac{C_T}{K}\}$ where 
\[
\mathcal{B}_K := \big\{(z,w): z_t - z_s \geq w_t - w_s - c_T(t-s)(1 +K),\ \forall 0\leq s\leq t\leq T\big\} .
\]
We claim that $\mathcal{A}$ is closed in $\mathcal{P}\left(\mathcal{C}([0, T])\times\mathcal{C}([0, T])\right)$.  Indeed, if $(\theta^k)_{k\geq1} \subset\mathcal{A}$ is a sequence converging in the weak sense to $\theta$, since $\mathcal{B}_K$ is clearly a closed subset of $\mathcal{C}([0, T])\times\mathcal{C}([0, T])$, it follows from the portmanteau theorem again that
 \[
1 - \frac{C_T}{K}\leq \limsup_k\theta^k(\mathcal{B}_K)\leq \theta(\mathcal{B}_K) \ \Rightarrow\ \theta \in \mathcal{B}_K.
 \]
 Thus $\mathcal{A}$ is indeed closed, so we can conclude that
  \[
1= \limsup_N \Xi_i^N (\mathcal{A}) \leq \Xi_i^\infty (\mathcal{A}) \ \Rightarrow\  \Xi_i^\infty (\mathcal{A}) =1.
 \]
 
 \vspace{0.3cm}
\noindent\textit{Step 2:} By Step 1, we see that for almost all $\theta$ under $\Xi^\infty_i$ and any $\varepsilon>0$, $k\geq1$, 
\begin{align}
\label{step2}
& \qquad \qquad \theta\left\{z: \tau_k(z)<T, \sup_{t\in[\tau_k(z), (\tau_k(z) + \varepsilon)\wedge T)}(z_t - k) = 0 \right\}\\
& \leq \theta\left\{z: \tau_k(z)<T, \sup_{t\in[\tau_k(z), (\tau_k(z) + \varepsilon)\wedge T)}[w_t - w_{\tau_k(z)} - c_T(t-\tau_k(z))\left(K+1\right)] = 0 \right\} + \frac{C_T}{K}\nonumber,
\end{align}
for any $K\geq0$.  The aim of this step is to show that the first term in the right-hand side of the above is zero.

We first claim that $w$ is a continuous martingale under $\theta$ for almost all $\theta$ under $\Xi^\infty_i$ (with respect to the canonical filtration $(\mathcal{F}_t)_{t\geq0}$ generated by $(z, w)$).  To this end, fix $s, t\in[0, T]$ with $s\leq t$.  Consider, for $N\geq i$,
\[
Q_{s, t}^{N,i} := \int \left( \langle \theta, (w_t - w_s)\varphi^s(z, w)\rangle\right)^2\Xi^N_i(d\theta)\rev{,}
\]
where $\varphi^s: \mathcal{C}([0, s])\times \mathcal{C}([0, s]) \to\R$ is a bounded continuous function.
Then by setting $Y_{s, t}^j := \int_s^t\sigma(Z^j_r - M^j_r)dW_r^j$, we have that
\begin{align*}
Q_{s, t}^{N,i} 
&=\frac{1}{(S^N_i)^2}\sum_{j=1}^NJ^2_{ij} \E\left[\left(Y_{s, t}^j\right)^2\left(\varphi^s\left(Z^j, Y_{0, \cdot}^j \right)\right)^2 \right]\\
&+  \frac{1}{(S^N_i)^2}\sum_{j=1}^N\sum_{k=1, k\neq j}^NJ_{ij}J_{ik}\E\left[ \varphi^s\left(Z^j, Y_{0, \cdot}^j \right) \varphi^s\left(Z^k, Y_{0, \cdot}^k \right)\E[Y_{s, t}^jY_{s, t}^k| \mathcal{F}_s] \right].
\end{align*}
By the independence of $(W^j)_{j\in\{1, \dots, N\}}$ (and the progressive measurability of $Z^j - M^j$ for all $j\in\{1, \dots, N\}$), the cross-terms are $0$, so that
\begin{align*}
Q_{s, t}^{N,i} 
&\leq \|\varphi^s\|^2_\infty\Lambda_\sigma^2(t-s)\frac{1}{(S^N_i)^2}\sum_{j=1}^NJ^2_{ij} \quad \to 0 
\end{align*}
as $N\to\infty$ by assumption.  Thus by the continuous mapping theorem \cite[Theorem 2.7]{Billingsley} again, we see that $\langle \theta, (w_t - w_s)\varphi^s(z, w)\rangle = 0$
almost surely, for almost all $\theta$ under $\Xi_i^\infty$.  This implies that $w$ is indeed a continuous martingale under $\theta$ for almost all $\theta$ under $\Xi^\infty_i$.

The second claim is that
\begin{equation}
\label{quad var}
t\Lambda_\sigma^{-2} \leq [w]_t \leq t\Lambda_\sigma^{2}, \quad t\in[0, T],
\end{equation}
almost surely under $\theta$, for almost all $\theta$ under $\Xi^\infty_i$ (where $[\cdot]$ indicates the quadratic variation).  To see this, let $B:= \left\{w\in\mathcal{C}([0, T]): \exists t\in[0, T]\ \mathrm{s.t.}\ [w]_t > \Lambda^2_\sigma t \right\}$ and $A = \{\theta: \theta(B) >0\}$.  The set $A$ is open in $\mathcal{P}\left(\mathcal{C}([0, T])\times\mathcal{C}([0, T])\right)$.  Indeed, if $(\theta^k)_{k\geq1} \subset A^c$ is a sequence converging in the weak sense to $\theta$, then since $B$ is clearly an open set in $\mathcal{C}([0, T])\times\mathcal{C}([0, T])$, we have by the portmanteau theorem that $0 = \liminf_k \theta^k(B) \geq \theta(B)$, so that $\theta(B) = 0$ $\Rightarrow$ $\theta \in A^c$ i.e. $A^c$ is closed. Then, since $A$ is open, by the portmanteau theorem once again, we see that $\liminf_N \Xi_i^N(A) \geq \Xi^\infty_i(A)$.  However, we know that for almost all $\theta$ under $\Xi_i^N$, 
\begin{align*}
\theta(B) &= \theta \left\{w\in\mathcal{C}([0, T]): \exists t\in[0, T]\ \mathrm{s.t.}\ [w]_t > \Lambda^2_\sigma t \right\}\\
&=\frac{1}{S_i^N}\sum_{j=1}^N\mathbb{P}\left(\left\{\exists t\in[0, T]\ \mathrm{s.t.}\ \int_0^t\sigma^2(Z^j_s - M^j_s)ds > \Lambda^2_\sigma t\right\}\right) = 0.
\end{align*}
Hence $0 = \Xi^\infty_i(A) =  \Xi^\infty_i\{\theta: \theta(B) >0\}$, which proves the right-hand side of  \eqref{quad var}.  The left-hand inequality follows similarly.


We can now use these two claims to conclude.  The point is that, since $w$ is a continuous martingale, we know that  $w$ is in fact a time-changed Brownian motion under $\theta$ for almost all $\theta$ under $\Xi^\infty_i$.  More precisely we have that under $\theta$, for almost all $\theta$ under $\Xi^\infty_i$,
\[
w_t = \tilde{w}_{[w]_t}, \qquad t\geq0,
\]
in law, where $\tilde{w}$ is a standard Brownian motion under $\theta$.  Returning now to \eqref{step2}, we then see that (since $\tau_k(z)$ is a stopping time for the filtration generated by $(z,w)$)
\begin{align*}
(w_t - w_{\tau_k(z)} - c_T(t-\tau_k(z))\left(K+1\right)) &\stackrel{d}{=} \tilde{w}_{[w]_t -[w]_{\tau_k(z)}} - c_T(t-\tau_k(z))\left(K+1\right),
\end{align*}
under $\theta$, for almost all $\theta$ under $\Xi^\infty_i$.  Using the standard properties of Brownian motion, it is then straightforward to see that, thanks to \eqref{quad var}, the first term in the right-hand side of \eqref{step2} is indeed $0$.

\vspace{0.3cm}
\noindent\textit{Step 3:}  In view of Step 2, by taking $K\to\infty$ in \eqref{step2} we see that
\[
 \theta\left\{z: \tau_k(z)<T, \sup_{t\in[\tau_k(z), (\tau_k(z) + \varepsilon)\wedge T)}(z_t - k) = 0 \right\} = 0,
\]
for almost all $\theta$ under $\Xi^\infty_i$, for any $i\in\mathbb{N}$.  To complete the proof of the lemma, note that $\Pi_i^N = \pi_{*}(\Xi_i^N)$ where $\pi(\theta)$ is defined to be the marginal of $\theta \in \mathcal{P}({\mathcal C}([0,T])\times{\mathcal C}([0,T]))$ on the first coordinate, and $\pi_*(\Xi_i^N)$ indicates the push-forward of $\Xi_i^N$ by $\pi$.  Thus by continuity, we have  $\Pi_i^\infty = \pi_{*}(\Xi_i^\infty)$.
 Then, for any Borel subset $A \subset {\mathcal C}([0,T])$,  $\int \theta \{(z_{t})_{t \in [0,T]} \in A\} \Xi^{\infty}_i(d\theta) = \int \mu \{(z_{t})_{t \in [0,T+1]} \in A\} \Pi^{\infty}_i(d\mu)$. Choosing $A = \{\tau^k  < T,  \sup_{t\in[\tau_k, (\tau_k + \varepsilon)\wedge T)} (z_{t} - k) =0\}$ completes the proof.  
\qed

\section{Conclusion}
\label{conclusion}
In this paper we have developed an original self-contained analysis of the convergence of a stochastic particle system that interacts through threshold hitting times, which is motivated by a new model of a neural network that attempts to take into account the dendritic structure of each neuron.  However, we argue that in fact the model is quite natural, and could be applied in other fields.  Indeed, the basic idea is that particles evolve independently in a very general way (they follow a stochastic differential equation with Lipschitz drift and bounded diffusion terms) until one of them reaches a threshold, at which point all the others feel the effect of this event.  In our model the effect of the hitting event on the other particles is smoothed by the kernel $G$, and it could be argued that this is natural in such physical systems, since typically the information about the occurrence of a hitting event may take some time to be spread throughout the system.  One specific example of a field where such models may be particularly useful is the study of the default rate of large portfolios (see also \cite{Giesecke1, Giesecke2}).  In a large portfolio of assets, a default event should quickly (though maybe not instantaneously) have an effect on all other assets in the portfolio.  

Finally we mention some of the open questions surrounding the model and the analysis we perform.  From a biological standpoint, the principal generalization we would like to make is to treat the case where the coupling chosen in Section \ref{sec: derivation} is of the form \eqref{form f wanted}, whereby the distribution of synapses on the dendritic tree may be inhomogeneous.  Such a question may require a quite different approach.
Another interesting direction would be to include a spatio-temporal noise term in equation \eqref{cable equation} modeling the transmission of the spike along the dendritic tree.  This would result in a stochastic partial differential equation, and a second source of noise.

\section{Appendix}

\subsection{Hitting time density bounds}
Fix $T>0$, and consider the stochastic differential equation
\begin{equation}
\label{SDE}
d\chi_t = (b(\chi_t) + \rev{\alpha}(t))dt + \sigma(\chi_t)dW_t, \quad t\in[0, T],
\end{equation}
where \rev{$\alpha$} is some function in \rev{$\mathcal{C}^1([0, T])$}.  We assume that  
\begin{enumerate}
\item $b:\R\to\R$ is continuously differentiable \rev{with bounded derivative, and set $\Lambda_b := \max\{\|b'\|_\infty, b(0)\}$ so that} $|b(z)| \leq \Lambda_b(1+|z|)$ and $|b'(z)| \leq \Lambda_b$ for all $z\in \R$;
\item $\sigma:\R\to\R$ is twice continuously differentiable and $\exists$ $\Lambda_\sigma >0$ such that $\Lambda_\sigma^{-1}\leq \sigma(z)\leq \Lambda_\sigma$ and $|\sigma'(z)|, |\sigma''(z)|\leq \Lambda_b$ for all $ z\in\R$.
\end{enumerate}
Under such conditions, it is well known that \eqref{SDE} has a unique strong solution.
We are interested in 
\[
\tau_1 = \inf\{t\geq0: \chi_{t\wedge T} \geq1\}, \ \  (\inf\emptyset = \infty) \quad \mathrm{and} \quad p^x(t): =  \frac{d}{dt}\mathbb{P}_x\left(\tau_1 \leq t\right).
\]
for $t\in [0, T]$ and $x<1$.  We will in fact write $p^x(t) = p^x_\rev{\alpha}(t)$ to emphasize the dependence on the function $\rev{\alpha}$, since our goal is to derive bounds for $p^x_\rev{\alpha}(t)$ in terms of $\rev{\alpha}$.

The following lemma provides a formula for the density $p^x_\rev{\alpha}(t)$, $t\in[0, T]$, $x<1$.  It is a generalization of the formula presented in \cite{Yor} for the Ornstein-Uhlenbeck process, and is obtained by the same calculation in our more general context.  In particular it is a slight simplification (as well as a generalization) of the main result in \cite{Pauwels}.

\begin{prop}
\label{Pauwels}
Under the preceding hypotheses, for $t\in[0,T]$, and $x<1$ we have that $p^x_\rev{\alpha}(t)$ is well-defined and is given by
\begin{equation}
\label{p}
p^x_\rev{\alpha}(t) = e^{h_\rev{\alpha}(t, S(1)) - h_\rev{\alpha}(0, S(x))}\varphi^{S(x)\to S(1)}_\rev{\alpha}(t)\frac{S(1) -S(x)}{\sqrt{2\pi t^3}}e^{-\frac{(S(1)-S(x))^2}{2t}},
\end{equation}
where $S(z) := \int_0^z[\sigma(y)]^{-1}dy$ for $z\in\R$ and
\begin{itemize}
\item $h_\rev{\alpha}(t, z) := \int_0^zB_\rev{\alpha}(t,y)dy$ for $t\in[0, T]$, $z\in \R$, where
\[
B_\rev{\alpha}(t, z) := \frac{b(S^{-1}(z))}{\sigma(S^{-1}(z))} - \frac{1}{2}\sigma'(S^{-1}(z)) + \frac{\rev{\alpha}(t)}{\sigma(S^{-1}(z))}, \quad t\in[0, T], z\in\R;
\]
\item \begin{align*}
\varphi^{S(x)\to S(1)}_\rev{\alpha}(t) &:=\E\exp\left(- \frac{1}{2}\int_0^tg_\rev{\alpha}(u, r^{S(x)\to S(1)}_u)du\right);
\end{align*}
\item $g_\rev{\alpha}(t, z) := B_\rev{\alpha}(t, z)^2 + 2\partial_th_\rev{\alpha}(t, z) +  \partial_zB_\rev{\alpha}(t, z)$ for $t\in[0, T]$ and $z\in \R$;
\item $(r^{S(x)\to {S(1)}}_u)_{u\in [0, t]}$ is a Bessel (3)-bridge from $(0, S(x))$ to $(t, S(1))$, i.e.
\[
r^{S(x)\to {S(1)}}_u := S(1) - \frac{t-u}{t}R^{S(1)-S(x)}\left(\frac{ut}{t-u}\right), \quad u\in [0, t),
\]
$r^{S(x)\to S(1)}_t := S(1)$, where $(R^y(\gamma))_{\gamma\geq0}$ denotes a Bessel (3)-process started at $y\geq0$.
\end{itemize}
\end{prop}

\rev{
\begin{proof}[Proof (sketch)]
The first step is to use the Lamperti transform $\tilde{\chi}_t := S(\chi_t)$, so that $d\tilde{\chi}_t = B_\alpha(t, \tilde{\chi}_t)dt + dW_t$.  We can then use the Girsanov theorem to write the Radon-Nikodym density of $\tilde{\mathbb{P}}_x = \mathrm{Law}(\tilde{\chi} |\tilde{\chi}_0 =x)$ with respect to the Wiener measure from $x$, ${\mathbb{W}}_x$.  By integrating by parts the stochastic integral appearing in the density and using It\^o's formula, we can express the Radon-Nikodym density on $\mathcal{F}_t$ as $Z_t = \exp(h_\alpha(t, \omega(t)) - h_\alpha(0, \omega(0)))\exp(-\frac{1}{2}\int_0^tg_\alpha(u, \omega(u))du)$ where $(\omega(t))_{t\in[0, T]}$ is the canonical process on $\mathcal{C}([0, T])$.
An application of Doob's optimal stopping theorem then yields
\[
\mathbb{P}_x(\tau_1 \in dt) = \E^{\mathbb{W}}_x\Big(Z_t|\tau_1 = t\Big)\mathbb{W}_x(\tau_1\in dt).
\]
To conclude we insert the known density $\mathbb{W}_x(\tau_1\in dt)$ and use the fact that the law of $\omega$ under $\mathbb{W}_0$ conditional on the event $\tau_{1-x} =t$ is given by the law of a Bessel (3)-bridge from $0$ to $1-x$.  For details see \cite{Yor} which simplifies the formula in \cite{Pauwels}.
\end{proof}
}

We first use Proposition \ref{Pauwels} to derive the following bound.
\begin{prop}
\label{density bound absolute}
For any $\rev{\alpha}\in \rev{\mathcal{C}^1([0, T])}$ there exists a constant $C_T$ depending only on $T$, $\Lambda_b$, $\Lambda_\sigma$, and \rev{$\|\alpha\|_{\mathcal{C}^1([0, T])}$} such that
\[
|p_{\rev{\alpha}}^x(t)| \leq C_Te^{C_Tx^2}(1-x)t^{-\frac{3}{2}} e^{-\frac{(1-x)^2}{2\Lambda^2_\sigma t}},
\]
for all  $t\in [0,T]$ and $x<1$.
\end{prop}

\begin{proof}[Proof (outline)]
The hypotheses on $b$, $\sigma$ and $\rev{\alpha}$ can be used to see that there exists a constant $C_T$ depending only on $T, \Lambda_b, \Lambda_\sigma$ and \rev{$\|\alpha\|_{\infty, T}$} such that
\begin{equation}
\label{expo bound}
e^{h_{\rev{\alpha}}(t, S(1)) - h_{\rev{\alpha}}(0, S(x)) }\leq C_Te^{C_Tx^2}, \quad x<1,\ t\in[0, T].
\end{equation}
Thanks to \eqref{p}, and the fact that \rev{$\Lambda^{-1}_\sigma \leq S' \leq \Lambda_\sigma$}, it thus suffices to  bound $\varphi^{S(x)\to S(1)}_\rev{\alpha}(t)$ for $x<1$ and $t\in[0, T]$.  To this end there exists $C_T$ depending only $\Lambda_b$, $\Lambda_\sigma$, $T$, and \rev{$\|\alpha\|_{\mathcal{C}^1([0, T])}$} that is allowed to change from line to line below such that
\begin{align*}
&\varphi^{S(x)\to S(1)}_\rev{\alpha}(t) \leq \E\exp\left(- \frac{1}{2}\int_0^t\left(2\partial_uh_\rev{\alpha}(u, r^{S(x)\to S(1)}_u) +  \partial_zB_\rev{\alpha}(u, r^{S(x)\to S(1)}_u)\right)du\right)\\
&\quad\leq C_T\E\exp\left(C_T\left\{\int_0^t |r^{S(x)\to S(1)}_u|du\right\}\right)\\
& \quad\leq C_T\E\exp\left(C_T\int_0^t \frac{t-u}{t}R^{S(1)-S(x)}\left(\frac{ut}{t-u}\right)du\right)\\
& \quad= C_T\E\exp\left(C_T t^3\int_0^\infty \frac{1}{(t+\gamma)^3}R^{S(1)-S(x)}\left(\gamma\right)d\gamma\right)
\end{align*}
where $\gamma = ut/(t-u)$ so that $du/d\gamma = t^2/(t+\gamma)^2$ and $(t-u)/t = t/(t+\gamma)$. Thus
\begin{align*}
\varphi^{S(x)\to S(1) }_\rev{\alpha}(t) &\leq  C_T\E\exp\left(C_T\sup_{\gamma \geq 0} \frac{R^{S(1)-S(x)}(\gamma)}{\sqrt{t+\gamma}}\right).
\end{align*}

Now $R^{S(1)-S(x)}(\gamma) \leq \sqrt{2}(S(1)-S(x) + |B^{(1)}_\gamma| + |B^{(2)}_\gamma| +|B^{(3)}_\gamma|)$ for all $\gamma \geq0$, where $B^{(1)}, B^{(2)},$ and  $B^{(3)}$ are independent standard 1-dimensional Brownian motions from $0$, by the definition of the Bessel $(3)$-process $(R^{S(1)-S(x)}(\gamma))_{\gamma\geq0}$.  Thus
\begin{align}
\label{varphi bound}
\varphi^{S(x)\to S(1)}_\rev{\alpha}(t) &\leq  C_Te^{C_T(1 - x)}\prod_{i=1}^3\E\exp\left(C_T\sup_{\gamma \geq 0} \frac{B^{(i)}_\gamma}{\sqrt{t+\gamma}}   \right)\nonumber\\
&=  C_Te^{C_T(1 - x)}\E\exp\left(C_T\sup_{\gamma \geq 0} B^{(1)}_{\frac{\gamma}{t + \gamma}}  \right)\nonumber\\
&=  C_Te^{C_T(1 - x)}\E\exp\left(C_T\sup_{\theta \in [0, 1]} B^{(1)}_{\theta}   \right) \leq C_Te^{C_T(1 - x)},
\end{align}
for all $t\in [0,T]$, $x<1$.
\end{proof}

We also have the following bound on the difference $|p_{{\rev{\alpha}}}^x - p_{\tilde{\rev{\alpha}}}^x|$:
\begin{prop}
\label{density bound diff}
For any \rev{$\alpha, \tilde{\alpha}\in \mathcal{C}^1([0, T])$} there exists a constant $C_T$ depending only on $T$, $\Lambda_b$, $\Lambda_\sigma$ and \rev{$\max\{\|\alpha\|_{\mathcal{C}^1([0, T])}$, $\|\tilde{\alpha}\|_{\mathcal{C}^1([0, T])}\}$} such that
\[
\left|p_{\rev{\alpha}}^x(t) - p_{\tilde{\rev{\alpha}}}^x(t)\right| \leq C_Te^{C_Tx^2}\|\rev{\alpha} - \tilde{\rev{\alpha}}\|_{\mathcal{C}^1([0, t])}(1 -x)t^{-\frac{3}{2}}e^{-\frac{(1-x)^2}{\Lambda_\sigma t}},
\]
for all  $t\in [0, T]$ and $x<1$.
\end{prop}

\begin{proof}[Proof (outline)]
Throughout the proof $C_T$ will be a constant depending only on $T, \Lambda_b, \Lambda_\sigma$ and \rev{$\max\{\|\alpha\|_{\mathcal{C}^1([0, T])}$, $\|\tilde{\alpha}\|_{\mathcal{C}^1([0, T])}\}$} that may change from line to line below.  In view of Proposition \ref{Pauwels}, we have
\begin{align}
&\left|p_{\rev{\alpha}}^x(t) -p_{\tilde{\rev{\alpha}}}^x(t) \right| \leq C_T(I_1 + I_2)\frac{1 -x}{\sqrt{2\pi t^3}}e^{-\frac{(1-x)^2}{\Lambda_\sigma t}},
\label{App: diff main}
\end{align}
where
\[
I_1 := e^{h_{\rev{\alpha}}(t, S(1)) - h_{\rev{\alpha}}(0, S(x)) }\left|D_1^{x, \rev{\alpha}, \tilde{\rev{\alpha}}}(t)\right|, \ 
D_1^{x,\rev{\alpha}, \tilde{\rev{\alpha}}}(t) := \varphi_{\rev{\alpha}}^{S(x)\to S(1)}(t)  - \varphi_{\tilde{\rev{\alpha}}}^{S(x)\to S(1)}(t)
\]
and
\[
I_2 :=  \varphi_{\tilde{\rev{\alpha}}}^{S(x)\to S(1)}(t)\left| D_2^{x, \rev{\alpha}, \tilde{\rev{\alpha}}}(t)\right|,\ 
D_2^{x,\rev{\alpha}, \tilde{\rev{\alpha}}}(t) := e^{h_{\rev{\alpha}}(t, S(1)) - h_{\rev{\alpha}}(0, S(x)) } -e^{h_{\tilde{\rev{\alpha}}}(t, S(1)) - h_{\tilde{\rev{\alpha}}}(0, S(x))}. 
\]

\vspace{0.2cm}
\noindent\textit{Bound for $I_2$:}  Without loss of generality, suppose that $h_{\rev{\alpha}}(t, S(1)) - h_{\rev{\alpha}}(0, S(x))  \geq h_{\tilde{\rev{\alpha}}}(t, S(1)) - h_{\tilde{\rev{\alpha}}}(0, S(x))$.  Then since $1 - e^{-x} \leq x$ for $x\geq0$,
\begin{align*}
&\left| D_2^{x,\rev{\alpha}, \tilde{\rev{\alpha}}}(t) \right| &\leq  e^{h_{\rev{\alpha}}(t, S(1)) - h_{\rev{\alpha}}(0, S(x)) }\left(h_{\rev{\alpha}}(t, S(1)) - h_{\tilde{\rev{\alpha}}}(t, S(1)) - h_{\rev{\alpha}}(0, S(x))  + h_{\tilde{\rev{\alpha}}}(0, S(x)) \right).
\end{align*} 
By definition of $h_\rev{\alpha}$,
\begin{align*}
&h_{\rev{\alpha}}(t, S(1)) - h_{\tilde{\rev{\alpha}}}(t, S(1)) - h_{\rev{\alpha}}(0, S(x))  + h_{\tilde{\rev{\alpha}}}(0, S(x))\\
&= (\rev{\alpha}(t) - \tilde{\rev{\alpha}}(t))\int_{S(x)}^{S(1)}\frac{1}{\sigma(S^{-1}(y))}dy + \int_0^t(\rev{\alpha}'(s) - \tilde{\rev{\alpha}}'(s))ds\int_0^{S(x)}\frac{1}{\sigma(S^{-1}(y))}dy\\
&\leq C_T(1 + |x|)\|\rev{\alpha} - \tilde{\rev{\alpha}}\|_{\mathcal{C}^1([0, T])}.
\end{align*}
Thus, using \eqref{expo bound} and \eqref{varphi bound}
\begin{align}
\label{I2 bound}
I_2&\leq  C_Te^{C_Tx^2}\|\rev{\alpha} - \tilde{\rev{\alpha}}\|_{\mathcal{C}^1([0, T])}, \quad t\in[0, T],\ x<1.
\end{align}

\vspace{0.2cm}
\noindent\textit{Bound for $I_1$:} 
We have for $t\in [0, T]$, $x<1$,
\begin{align*}
&\left|D_1^{x,\rev{\alpha}, \tilde{\rev{\alpha}}}(t)\right| \leq \E\left|e^{- \frac{1}{2}\int_0^tg_{\rev{\alpha}}(u, r^{S(x)\to S(1)}_u)du}  - e^{- \frac{1}{2}\int_0^tg_{\tilde{\rev{\alpha}}}(u, r^{S(x)\to S(1)}_u)du}\right|\\
&\leq \frac{1}{2}\E\Big[e^{\frac{1}{2}\max\{\int_0^t|g_{\rev{\alpha}}(u, r^{S(x)\to S(1)}_u)|du,\int_0^t|g_{\rev{\alpha}}(u, r^{S(x)\to S(1)}_u)|du\}}\\
&\qquad\qquad\qquad\qquad \times \int_0^t|g_{\rev{\alpha}}(u, r^{S(x)\to S(1)}_u) -g_{\tilde{\rev{\alpha}}}(u, r^{S(x)\to S(1)}_u)|du\Big]\\
&\leq C_Te^{C_T(1-x)}\left(\E\Big[\int_0^t|g_{\rev{\alpha}}(u, r^{S(x)\to S(1)}_u) -g_{\tilde{\rev{\alpha}}}(u, r^{S(x)\to S(1)}_u)|^2du\Big]\right)^\frac{1}{2}
\end{align*}
by using Cauchy-Schwarz and \eqref{varphi bound} to bound the first factor. Now by definition,
\begin{align*}
&\left|g_{\rev{\alpha}}(u, z) -g_{\tilde{\rev{\alpha}}}(u, z)\right|\leq C_T(1+|z|)\|\rev{\alpha} - \tilde{\rev{\alpha}}\|_{\mathcal{C}^1([0, T])},
\end{align*}
for all $u\in [0, t]$, and $z\in\R$.  Therefore, by \eqref{expo bound},
\begin{align*}
I_1 &\leq C_Te^{C_Tx^2}\left(1 + \E\Big[\int_0^t(r^{S(x)\to S(1)}_u)^2du\Big]\right)^\frac{1}{2}\|\rev{\alpha} - \tilde{\rev{\alpha}}\|_{\mathcal{C}^1([0, T])}.
\end{align*}
To complete the proof, it thus suffices to bound
\begin{align*}
\E\left[\int_0^t(r^{S(x)\to S(1)}_u)^2du \right] &= \E\left[\int_0^t\left( S(1) - \frac{t-u}{t}R^{S(1)-S(x)}\left(\frac{ut}{t-u}\right)\right)^2du \right]\\
&\leq 2(S(1))^2t + 2\int_0^\infty\frac{t^4}{(t+\gamma)^4}\E\left[\left(R^{S(1)-S(x)}\right)^2\left(\gamma\right)\right]d\gamma,
\end{align*}
where $\gamma = ut/(t-u)$ so that $du/d\gamma = t^2/(t+\gamma)^2$ and $(t-u)/t = t/(t+\gamma)$ as before.  Now using the simple bound $\E[(R^{S(1)-S(x)})^2(\gamma)] \leq 4	((S(1)-S(x))^2 + \gamma)$ \rev{(recall that by definition 
$(R^y)^2(\gamma) := (B^{(1)}_\gamma + y)^2 + (B^{(2)}_\gamma)^2 + (B^{(3)}_\gamma)^2$ for any $y\geq0$ and $\gamma\geq0$, where $B^{(1)}, B^{(2)}$ and $B^{(3)}$ are independent 1-dimensional Brownian motions from $0$)}, we can conclude.
\end{proof}

\subsection{Improvements.}
\label{Improvements}
It should be noted that the bounds in Propositions \ref{density bound absolute} and \ref{density bound diff}, although good for small time, are a long way from being optimal.  In particular the factor of $\exp(C_Tx^2)$ in both bounds is rather unsatisfactory, since for large time it is not dominated by the term $\exp(-(1-x)^2/\Lambda_\sigma t)$.  It is for this reason that in the above work (see Assumptions \ref{assumptions}), we have assumed that the initial condition $U_0$ has compact support.  Although not critically important for the main thrust of this article, the prospect of having a better density estimate is an interesting question.

Indeed, suppose that we are again in the situation of Proposition \ref{density bound absolute}.  Without loss of generality we may suppose that $\sigma\equiv1$ (otherwise we may use the Lamperti transform $S$ as above).  We would in fact like to prove that
\begin{equation}
\label{key quantity}
p_{\rev{\alpha}}^x(t) = e^{h_\rev{\alpha}(t, 1) - h_\rev{\alpha}(0, x)}\varphi^{x\to 1}_\rev{\alpha}(t) \frac{1-x}{\sqrt{2\pi t^3}}  e^{-\frac{(1-x)^2}{2t}} \leq C_T(1-x)t^{-\frac{3}{2}}  e^{-\frac{(1-x)^2}{C_Tt}}
\end{equation}
for $x<1$, $t\in(0, T]$ and some constant $C_T$.
To this end, we look again at the left hand side of \eqref{key quantity}.  The hypotheses on $b$ and $\rev{\alpha}$ can be used to see that
\begin{equation*}
e^{h_{\rev{\alpha}}(t, S(1)) - h_{\rev{\alpha}}(0, S(x)) }\leq C_Te^{C_T|x| + \int_x^{1}b(y)dy}. 
\end{equation*}
Moreover we have
\begin{align}
\label{phi1}
\varphi^{x\to 1}_\rev{\alpha}(t) 
&\leq C_T\E\exp\left( - \frac{1}{2}\int_0^tb^2(r^{x\to 1}_u)du + C_T\int_0^t |r^{x\to 1}_u|du\right).
\end{align}
Using these two estimates in \eqref{key quantity} yields (after substituting in the definition of $r^{x\to 1}_u$ and using the simple bound\footnote{Recall that the Bessel $3$-process $R^{z}$ starting from $z\geq0$ is defined as
\[
R^{z}(\gamma) := \sqrt{(z + B^{(1)}_\gamma)^2 + (B^{(2)}_\gamma)^2 + (B^{(3)}_\gamma)^2}, \quad, \gamma\geq0,
\]
where $B^{(1)}, B^{(2)},$ and  $B^{(3)}$ are independent standard 1-dimensional Brownian motions.}
$R^{1-x}(\gamma) \leq \sqrt{2}( 1- x +  R^{0}(\gamma))$ for all $x<1$ and $\gamma\geq 0$)
\begin{align*}
\label{total1}
p_{\rev{\alpha}}^x(t)&\leq  C_Te^{C_T|x|} \E \left(e^{ \int_x^{1}b(y)dy - \frac{1}{2}\int_0^tb^2(r^{x\to 1}_u)du -\frac{(1-x)^2}{2t}  + C_T\int_0^t \frac{t-u}{t}R^{0}\left(\frac{ut}{t-u}\right)du}\right)(1-x)t^{-\frac{3}{2}}\nonumber\\
&\leq C_Te^{C_T|x|} \left[\E \left(e^{ 2\int_x^{1}b(y)dy - \int_0^tb^2(r^{x\to 1}_u)du -\frac{(1-x)^2}{t}}\right)\right]^\frac{1}{2}(1-x)t^{-\frac{3}{2}}.
\end{align*}
For the second inequality we have used the Cauchy-Schwarz inequality and the bound on the expectation of the exponential of $\int_0^t \frac{t-u}{t}R^{0}(\frac{ut}{t-u})du$ proved at the end of Proposition \ref{density bound absolute}.  It is thus clear that if we can prove a bound of the form
\begin{equation}
\label{hope}
J(t, x):=\E \left(e^{ 2\int_x^{1}b(y)dy - \int_0^tb^2(r^{x\to 1}_u)du -\frac{(1-x)^2}{t}}\right) \leq C_Te^{C_T|x|}e^{-\frac{(1-x)^2}{C_Tt}}, \quad \forall t\in(0, T],
\end{equation}
for some $C_T$, at least when $x$ is sufficiently negative, then \eqref{key quantity} will follow.  

The hope is that the term $2\int_x^{1}b(y)dy$ in $J(t, x)$, which is potentially of order $x^2$, may be compensated by a combination of $-\int_0^tb^2(r^{x\to 1}_u)du$ and $-(1-x)^2/t$.  Whether or not this is true for all Lipschitz $b$ is not clear to us.  However, there are some examples when \eqref{hope} (and hence \eqref{key quantity}) can be easily shown to hold:
\begin{itemize}
\item [(i)] Suppose that $b$ is bounded from above i.e. $b(x) \leq K$ for all $x\leq1$ (note $b$ is not necessarily bounded from below). Then clearly we have that
\[
J(t, x) 
\leq C_Te^{C_T|x|}e^{-\frac{(1-x)^2}{t}},
\]
for all $t\in(0, T]$, where $C_T$ depends on $K$.
\item  [(ii)] Suppose that $b(x) = -\lambda x$ with $\lambda >0$ (the case $\lambda \leq 0$ is covered by (i)).  Then
\begin{align*}
J(t, x) &=\E \exp\left(-\lambda (1 - x^2) - \lambda^2\int_0^t(r^{x\to 1}_u)^2du -\frac{(1-x)^2}{t}\right)\\
&\leq \E \exp\left(\lambda (1-x)^2 - \lambda^2\int_0^t\left[1 - \frac{t-u}{t}R^{1-x}\left(\frac{ut}{t-u}\right)\right]^2du -\frac{(1-x)^2}{t}\right)\\
&\leq C_Te^{C_T|x|}\exp\left( -\frac{(1-x)^2}{t}\left[\frac{1}{3}\lambda^2t^2 - \lambda t +1\right]\right).
\end{align*}
Finally note that $\frac{1}{3}w^2 - w + 1 \geq \frac{1}{4}$ for all $w\in\R$. Therefore 
\begin{align*}
J(t, x) &\leq C_Te^{C_T|x|}\exp\left( -\frac{(1-x)^2}{4t}\right).
\end{align*}
\item [(iii)] By combining the above two cases, we can see that \eqref{hope} also holds when $b(x) = -\lambda x + h(x)$ where $\lambda \in \R$ and $h$ is bounded from above on $(-\infty, 1]$.
\end{itemize}

\subsection{Nomenclature}
\label{nomenclature}

\begin{center}
\begin{tabular}{ | c | c |}
\hline
 \hspace {0.5cm}\textbf{Notation:}\hspace{0.5cm} &  \hspace{0.5cm}\textbf{Defined in:}\hspace{0.5cm}\\ \hline
 \hspace {0.5cm}$U_t^i,\ M_t^i,\ \tau_k^i,\ S_i^N$\hspace{0.5cm} & \hspace{0.5cm}\eqref{particle system}\hspace{0.5cm}\\
 \hspace {0.5cm}$U_t,\ M_t,\ \tau_k$\hspace{0.5cm} & \hspace{0.5cm}\eqref{limit equation}\hspace{0.5cm}\\
 \hspace {0.5cm}$Z_t$\hspace{0.5cm} & \hspace{0.5cm}\eqref{Z limit equation}\hspace{0.5cm}\\
 \hspace {0.5cm}$\Phi(h),\ M_t^h,\ Z_t^h,\ f_h(t),\ \tau_k^h$\hspace{0.5cm} & \hspace{0.5cm}\eqref{Phi}--\eqref{fh}\hspace{0.5cm}\\
 \hspace {0.5cm}$f^{\sharp s}_h(r)$\hspace{0.5cm} & \hspace{0.5cm}\eqref{fshift}\hspace{0.5cm}\\
 \hspace {0.5cm}$\bar{\mu}^N_i,\ \Pi_i^N$\hspace{0.5cm} & \hspace{0.5cm}\eqref{mu bar}-\eqref{Pi def}\hspace{0.5cm}\\
 \hspace {0.5cm}$m_t,\ m_t(z),\ n_t,\ n_t(z)$\hspace{0.5cm} & \hspace{0.5cm}\eqref{mn}\hspace{0.5cm}\\
 \hspace {0.5cm}$\mathcal{L}_{t, \mu}$\hspace{0.5cm} & \hspace{0.5cm}\eqref{generator}\hspace{0.5cm}\\
 \hline
\end{tabular}
\end{center}


\end{document}